\documentclass[reqno]{amsart}
\usepackage{amssymb,amsmath,amsthm}
\usepackage{color}
\usepackage{enumerate}
\usepackage{fullpage}
\usepackage{url}
\usepackage[Symbol]{upgreek}
\usepackage[unicode,pdfusetitle]{hyperref}
\usepackage[mathscr]{euscript}

\numberwithin{equation}{section}
\newtheorem{lemma}{Lemma}[section]
\newtheorem{thm}[lemma]{Theorem}
\newtheorem{corollary}[lemma]{Corollary}

\newtheorem{proposition}[lemma]{Proposition}
\newtheorem{open problem}[lemma]{Open problem}

\newtheorem*{fact*}{Fact}
\newtheorem{fact}[lemma]{Fact}

\newtheorem*{claim*}{Claim}

\theoremstyle{definition}

\newtheorem{definition}[lemma]{Definition}

\newtheorem{remark}[lemma]{Remark}
\newtheorem{final remark}[lemma]{Final remark}
\newtheorem{example}[lemma]{Example}

\newtheorem{question}[lemma]{Question}

\DeclareMathOperator{\an}{an}
\DeclareMathOperator{\anexp}{an,exp}
\DeclareMathOperator{\cl}{c\ell}
\DeclareMathOperator{\CODF}{CODF}
\DeclareMathOperator{\dcl}{dcl}
\DeclareMathOperator{\dense}{dense}

\DeclareMathOperator{\fr}{fr}

\DeclareMathOperator{\id}{id}

\DeclareMathOperator{\ord}{ord}
\DeclareMathOperator{\pred}{Pr}
\DeclareMathOperator{\RCF}{RCF}

\DeclareMathOperator{\ring}{ring}
\DeclareMathOperator{\rk}{rk}

\DeclareMathOperator{\tcl}{cl}

\DeclareMathOperator{\tp}{tp}

\newcommand{\M}{\mathbb{M}}
\newcommand{\N}{\mathbb{N}}
\newcommand{\bP}{\mathbb{P}}
\newcommand{\Q}{\mathbb{Q}}
\newcommand{\R}{\mathbb{R}}
\newcommand{\Z}{\mathbb{Z}}

\newcommand{\cC}{\mathcal C}
\newcommand{\cD}{\mathcal D}
\newcommand{\cM}{\mathcal M}
\newcommand{\cN}{\mathcal N}
\newcommand{\cP}{\mathcal P}
\newcommand{\cR}{\mathcal R}
\newcommand{\cV}{\mathcal V}

\newcommand{\av}{{\bar a}}
\newcommand{\bv}{{\bar b}}
\newcommand{\cv}{{\bar c}}

\newcommand{\sv}{{\bar s}}
\newcommand{\uv}{{\bar u}}
\newcommand{\vv}{{\bar v}}

\newcommand{\x}{{\bar x}}
\newcommand{\y}{{\bar y}}

\newcommand{\zero}{{\bar 0}}

\newcommand{\0}{\emptyset}
\newcommand{\bi}{\boldsymbol{i}}
\newcommand{\cld}{\cl^\delta}
\newcommand{\clD}{\cl^\Delta}
\newcommand{\clDz}{\cl^{\Delta_0}}
\newcommand{\D}{\operatorname D\!}
\newcommand{\dclL}{\dcl_L}
\newcommand{\Der}{\mathfrak{Der}}
\newcommand{\dif}{\operatorname d\!}
\newcommand{\dimd}{\dim^\delta}
\newcommand{\dimD}{\dim^\Delta}
\newcommand{\dimL}{\dim_L}

\newcommand{\Jac}{\mathbf{J}}
\newcommand{\jet}{\mathscr{J}_\delta}
\newcommand{\Ld}{L^\delta}
\newcommand{\LD}{L^\Delta}
\newcommand{\rkd}{\rk^\delta}
\newcommand{\rkD}{\rk^\Delta}
\newcommand{\rkL}{\rk_L}
\newcommand{\Td}{T^\delta}
\newcommand{\TD}{T^\Delta}
\newcommand{\TdG}{\Td_G}
\newcommand{\TDG}{\TD_G}

\renewcommand{\preceq}{\preccurlyeq}
\renewcommand{\succeq}{\succcurlyeq}
\renewcommand{\epsilon}{\varepsilon}

\providecommand{\noopsort}[1]{}
\providecommand{\bysame}{\leavevmode\hbox to3em{\hrulefill}\thinspace}
\providecommand{\MR}{\relax\ifhmode\unskip\space\fi MR }

\providecommand{\href}[2]{#2}

\author{Antongiulio Fornasiero}
\author{Elliot Kaplan}
\email{antongiulio.fornasiero@gmail.com}
\email{eakapla2@illinois.edu}
\title{Generic derivations on o-minimal structures}
\address{Dipartimento di Matematica e Informatica ``Ulisse Dini,'' Viale Morgagni, 67/a, 50134 Firenze, Italy}
\address{Department of Mathematics, University of Illinois at Urbana-Champaign, Urbana, IL 61801}
\date{\today}
\begin{document}

\begin{abstract}
Let $T$ be a complete, model complete o-minimal theory
extending the theory $\RCF$ of real closed ordered fields in some appropriate
language $L$. We study derivations $\delta$ on models $\cM\models T$. We introduce the notion of a $T$-derivation: a derivation which is compatible with the $L(\0)$-definable $\mathcal{C}^1$-functions on $\cM$. We show that the theory of $T$-models with a $T$-derivation has a model completion $T^\delta_G$. The derivation in models $(\cM,\delta)\models T^\delta_G$ behaves ``generically,'' it is wildly discontinuous and its kernel is a dense elementary $L$-substructure of $\cM$. If $T = \RCF$, then $T^\delta_G$ is the theory of closed ordered differential fields (CODF) as introduced by Michael Singer. We are able to recover many of the known facts about CODF in our setting. Among other things, we show that $T^\delta_G$ has $T$ as its open core, that $T^\delta_G$ is distal, and that $T^\delta_G$ eliminates imaginaries. We also show that the theory of $T$-models with finitely many commuting $T$-derivations has a model completion.
\end{abstract}

\maketitle

\setcounter{tocdepth}{1}
\tableofcontents
\section{Introduction}

Let $\cM = (M; +, \cdot, 0,1, \dots)$ be a structure expanding a field. A \textbf{derivation} on $\cM$ is a function $\delta: M \to M$ such that:
\[
\delta(x + y)\ =\ \delta x + \delta y,\qquad \delta(xy)\ =\ x \delta y+ y \delta x.
\]
If $\cM$ has extra structure besides the field operations, the derivation may have nothing to do with this extra structure. For instance, if $\cM = \R_{\exp}:=(\R; +, \cdot, 0, 1, \exp)$ then it may not be the case that $\delta\exp(x) \neq\exp(x) \delta x$.

\medskip

In this paper, we consider the case when $\cM$ is an o-minimal structure (expanding a real closed ordered field) and we study derivations on $\cM$ which are \emph{compatible} with the structure $\cM$ in the following sense: for every $\cC^1$-function $f: M^{n} \to M$ which definable in $\cM$ without parameters, we require that:
\[
\delta f(\x)\ =\ \sum_{i=1}^{n} \frac{\partial f}{\partial x_{i}}(\x)\, \delta x_{i}.
\]
For instance, when $\cM = \R_{\exp}$, we impose, among other conditions, that $\delta \exp(x) =\exp(x) \delta x$.
We show that if $\cM$ is a pure ordered field then every derivation on $\cM$ is
already compatible with $\cM$ (Proposition~\ref{semialg}). 

\medskip

There are many natural examples where the above compatibility condition is met. For instance, it is well known that the germs at infinity of unary functions definable in any o-minimal expansion $\widetilde{\R}$ of the real field form a Hardy field $H(\widetilde{\R})$ with a natural derivation $\frac{\dif}{\dif x}$. There is a natural way to expand $H(\widetilde{\R})$ so that $H(\widetilde{\R}) \equiv \widetilde{\R}$, and $\frac{\dif}{\dif x}$ is compatible with this expansion. Another natural example is the ordered field $\mathbb{T}^{LE}$ of logarithmic-exponential transseries (see~\cite{ADH17}). There is a natural expansion of $\mathbb{T}^{LE}$ which makes it an elementary extension of the real field with restricted analytic functions and an exponential function. The natural derivation on $\mathbb{T}^{LE}$ is compatible with this expansion.

\medskip

In~\cite{Si78}, Singer showed that the theory of ordered differential fields has a model completion: the theory of closed ordered differential fields ($\CODF$). He provided an axiomatization of $\CODF$ and proved that it has quantifier elimination. Since then, many others have contributed to the model theory of closed ordered differential fields. Among these contributions is a cell-decomposition theorem and a corresponding dimension function~\cite{BMR09} as well as a proof that $\CODF$ has o-minimal open core and eliminates imaginaries~\cite{Po11}.

\medskip

Let $T$ be a model-complete o-minimal theory in some language $L$ and set $\Ld := L\cup\{\delta\}$. Let $\Td$ be the $\Ld$-theory which extends $T$ by axioms asserting that $\delta$ is compatible in the way defined above. In \S\ref{sec:TdG}, we show that $\Td$ has a model completion which we denote by $\TdG$. A simple axiomatization for $\TdG$ extends $\Td$ by the following axiom scheme:
\begin{itemize}
\item[(G)] If $X \subseteq M^{n \times n}$ is $L(M)$-definable and the projection onto the first $n$ coordinates has nonempty interior, then there exists $\av \in M^{n}$ such that $(\av, \delta \av) \in X$.
\end{itemize}
We go on to explore the properties of $\TdG$: we show that it does not have prime models in general, it is NIP and distal, and it is not strongly dependent. We end \S\ref{sec:TdG} by showing $\RCF^\delta_G = \CODF$ and that $\TdG$ can be seen as a distal extension of the theory of dense pairs of models of $T$, which is not itself distal~\cite{HN17}. The fact that $\CODF$ itself is a distal extension of the theory of dense pairs of real closed ordered fields was first established by Cubides Kovacsics and Point~\cite{CP19}.

\medskip

In \S\ref{sec:geom}, we show that any model $(\cM, \delta) \models \TdG$ has a (unique) dimension function in the sense of~\cite{vdD89} and a kind of cell decomposition in the sense of~\cite{BMR09}. We go on to show that $T$ is the open core of $\TdG$, (that is, every open $\Ld(M)$-definable subset of $M^{n}$ is already $L(M)$-definable). We use this open core result to show that $\TdG$ eliminates imaginaries and that it eliminates the quantifier $\exists^\infty$.

\medskip

The final section, \S\ref{sec:multi} is dedicated to the study of several commuting derivations which are compatible with models of $T$. The model completion of the theory of fields of characteristic zero with several commuting derivations was axiomatized by McGrail~\cite{McG00} and the model completion of the theory of ordered fields
with several commuting derivations was axiomatized by Rivi\`ere~\cite{Ri06} (see also~\cite{Tr05}). Let $\Delta = \{\delta_{1}, \dots, \delta_{p}\}$ be a finite set of derivations and let $\TD$ be the $L \cup\Delta$-theory which asserts that each $\delta_i$ is compatible and that each $\delta_i$ and $\delta_j$ commute. We show that $\TD$ has a model completion $\TDG$. The main difficulty is giving an axiomatization for $\TDG$: while the axiom scheme (G) for one derivation is quite simple, an axiomatization for $\TDG$ is quite complicated when $n \ge 2$. We go on to define a dimension function in models of $\TDG$ and we show that $\TDG$ has $T$ as its open core.

\medskip

In Appendix~\ref{sec:Ck}, we use work of Loi~\cite{Lo98} to prove a result about $\cC^{k}$-functions definable in o-minimal structures. This result, Corollary~\ref{Ckfiber}, generalizes a known fact about definable continuous functions and it may be of independent interest. We initially proved our main results without this corollary, but it does simplify things significantly, especially in \S\ref{sec:multi}.
Also of independent interest may be \S\ref{subsec:quasi-endo} on ``quasi-endomorphisms'' of a finitary matroid.

\subsection{Notation and conventions}\label{sec:notation}
In this article, $T$ denotes a complete, model complete o-minimal theory
extending the theory $\RCF$ of real closed ordered fields in some appropriate
language $L$. We always use $\cM$ and $\cN$ to denote models of $T$ and we use
$M$ and $N$ to denote the underlying sets of $\cM$ and $\cN$.

\medskip

We will always use $k,m,$ and $n$ to denote non-negative integers. We view tuples in $M^n$ as column vectors and if $\av \in M^n$ and $\bv \in M^m$, we use $(\av,\bv)$ to denote the column vector $\binom{\av}{\bv} \in M^{n+m}$. We view elements of $M^{m\times n}$ as matrices and if $A \in M^{m\times n}$ and $\bv \in M^{n}$, we let $A\bv \in M^{m}$ be the usual product of the matrix $A$ and the vector $\bv$.

\medskip

Let $A \subseteq M$ and $D \subseteq M^n$. We say that $D$ is \textbf{$L(A)$-definable} if there is some $(m+n)$-ary $L$-formula $\varphi(\x,\y)$ and some tuple $\av \in A^m$ such that 
\[
D\ =\ \big\{\y \in M^n:\cM \models \varphi(\av,\y)\big\}.
\]
Given a function $f:D \to M$, we let $\Gamma(f)\subseteq M^{n+1}$ denote the graph of $f$ and we say that $f$ is $L(A)$-definable if $\Gamma(f)$ is. Note that this means that $D$ is also $L(A)$-definable. Given $k \leq n$, we denote the projection of $D$ onto the first $k$ coordinates by $\pi_k(D)$. For $\bv \in M^k$, we let $D_{\bv}$ denote the set $\big\{\y \in M^{n-k}:(\bv,\y) \in D\big\}$ and we let $f_{\bv}:D_{\bv} \to M$ denote the function $\y \mapsto f(\bv,\y)$.

\medskip

We let $\dclL(A)$ be the the definable closure of $A$ (in $\cM$, implicitly, but this doesn't change if we pass to elementary extensions of $\cM$). If $b \in \dclL(A)$, then there is an $L(\0)$-definable function $f:M^n \to M$ and a tuple $\av \in A^n$ such that $b = f(\av)$. A set $B \subseteq M$ is said to be \textbf{$\dclL(A)$-independent} or \textbf{$\dclL$-independent over $A$} if $b\not\in \dclL\!\big(A\cup( B \setminus \{b\})\big)$ for all $b \in B$. A tuple $(b_i)_{i \in I}$ is said to be $\dclL(A)$-independent if its set of components $\{b_i:i \in I\}$ is $\dclL(A)$-independent and if no components are repeated. It is well-known that $(\cM,\dclL)$ is a finitary matroid (also called a \emph{pregeometry}). We let $\rkL$ be the cardinal-valued rank function associated to this finitary matroid. 

\medskip

The cell decomposition theorem gives rise to a well-behaved dimension function on $L(M)$-definable sets: for an $L(M)$-definable set $A\subseteq M^n$, we let
\[
\dimL(A)\ :=\ \max\big\{i_1+\cdots+i_n: A\text{ contains a cell of type }(i_1,\ldots,i_n)\big\}.
\]
This dimension interacts nicely with the rank $\rkL$: let $\av \in M^n$ and $B \subseteq M$. If $\rkL(\av|B) = m\leq n$ then $\av$ is contained in some $L(B)$-definable set of dimension $m$ and $\av$ is not contained in any $L(B)$-definable set of dimension $<m$.

\medskip

We say that $\cM \preceq_L \cN$ if $\cM$ is an elementary $L$-substructure of $\cN$. For a subset $A \subseteq N$, we denote by $\cM\langle A\rangle$ the substructure of $\cN$ with underlying set $\dclL(M\cup A)$. As $T$ has definable Skolem functions, $\cM\langle A \rangle$ is an elementary substructure of $\cN$. We say that $A$ is a \textbf{basis for $\cN$ over $\cM$} if $A$ is $\dclL(M)$-independent and $\cN = \cM\langle A \rangle$. If $A = \{a_1,\ldots,a_n\}$, we write $\cM\langle a_1,\ldots,a_n\rangle$ instead of $\cM\langle A \rangle$. Given an $L(M)$-definable set $D \subseteq M^n$, we let $D^{\cN}$ denote the subset of $N^n$ defined by the same formula as $D$. Since $\cM \preceq_L\cN$, the set $D^{\cN}$ does not depend on the choice of defining formula. We sometimes refer to this as the \emph{natural extension} of $D$ to $\cN$ and we drop the superscript when it is clear from context. If $f:D \to M$ is an $L(M)$-definable function, then we let $f^{\cN}:D^{\cN}\to N$ be the $L(M)$-definable function with graph $\Gamma(f^{\cN}) = \Gamma(f)^{\cN}$.

\medskip

Since $T$ has definable Skolem functions, $T$ has a prime model which we denote by $\bP$. This prime model is always canonically isomorphic to $\dclL(\0)$ in any model of $T$. By a \textbf{monster model} of $T$, we mean a $\kappa$-saturated and strongly $\kappa$-homogeneous model of $T$ for some $\kappa> |T|:= \max\{|L|,\omega\}$. When working in a monster model, we use \emph{small} to mean of cardinality $<\kappa$.

\medskip

If $L'$ is another language, then we use the conventions above where they make sense. For example, $\dcl_{L'}$ will be the definable closure operator in a given $L'$-structure.

\medskip

An $L(M)$-definable function $f:D\to M$ with $D \subseteq M^n$ is said to be a \textbf{$\cC^k$-function} if there is an $L(M)$-definable open $U \supseteq D$ and an $L(M)$-definable $\cC^k$-function $F:U \to M$ with $F|_D = f$. This extension is not unique. A map $g= (g_1,\ldots,g_m):D \to M^m$ is said to be a $\cC^k$-map if each $g_i$ is a $\cC^k$-function. If $g$ is a $\cC^1$-function of $\y = (y_1,\ldots,y_n)$, then we let $\Jac_{g}$ denote the Jacobian matrix
\[
\Jac_{g}\ :=\ \left(\frac{\partial g_i}{\partial y_j}\right)_{1\leq i \leq m,1\leq j\leq n}
\]
viewed as a function from $D$ to $M^{m\times n}$. We also denote this function by $\frac{\partial g}{\partial \y}$ if we want to indicate the dependence on the variables.

\subsection{Acknowledgments}
We would like to thank Marcus Tressl, Lou van den Dries, Erik Walsberg and Itay Kaplan for helpful conversations and suggestions. 
We would also like to thank Angela Borrata for pointing out an error in an earlier version of this paper, and the anonymous referee for suggesting improvements.

\medskip

The first author began working on this project after attending the Workshop on Tame Expansions of O-minimal Structures at Konstanz from October 1-4, 2018. He is supported in part by GNSAGA of INdAM and by PRIN 2017, project 2017NWTM8R, Mathematical Logic, models, sets, computability.
\section{\texorpdfstring{$T$}{T}-derivations} \label{sec:Td}
In this section, we fix a map $\delta:M \to M$. Given a tuple $\x = (x_1, \dotsc, x_n) \in M^n$, we denote by $\delta(\x)$ the tuple $\big(\delta(x_1), \dotsc, \delta(x_n)\big)$. We often use $\delta \x$ instead of instead of $\delta(\x)$. We let $\Ld$ be the language $L \cup \{ \delta\}$ and we view $(\cM,\delta)$ as an $\Ld$-structure.

\subsection{$T$-derivations}
Given an $L(\0)$-definable $\cC^1$-function $f:U \to M$ with $U \subseteq M^n$ open, we say that $\delta$ is \textbf{compatible with $f$} if we have
\[
\delta f(\uv)\ =\ \Jac_f(\uv) \delta \uv
\]
for each $\uv\in U$.
If $g:U \to M^m$ is an $L(\0)$-definable $\cC^1$-map, we say that $\delta$ is compatible with $g$ if 
\[
\delta g(\uv) \ =\ \Jac_{g}(\uv)\delta \uv
\]
for each $\uv\in U$ or, equivalently, if $\delta$ is compatible with each component function $g_i$.

\begin{definition}
We say that $\delta$ is a \textbf{$T$-derivation} (on $\cM$) if $\delta$ is compatible with every $L(\0)$-definable $\cC^1$-function with open domain. Let $\Td$ be the $\Ld$-theory which extends $T$ by axioms stating that $\delta$ is a $T$-derivation. That is, $(\cM,\delta)\models \Td$ if and only if $\cM \models T$ and $\delta$ is a $T$-derivation on $\cM$.
\end{definition}

To justify the use of the name \emph{$T$-derivation}, recall that $\delta$ is a \textbf{derivation} (on $\cM$) if $\delta(x+y) = \delta x+\delta y$ and if $\delta (xy) = x\delta y +y\delta x$ for all $x,y \in M$.

\begin{lemma}\label{isader}
Any $T$-derivation is a derivation.
\end{lemma}
\begin{proof}
Use that $\delta$ is compatible with the functions $(x,y) \mapsto x+y$ and $(x,y) \mapsto xy$.
\end{proof}

It is a well-known fact that if $(K,\delta)$ is a differential field, then $\ker(\delta)= \big\{a \in K:\delta(a) = 0\big\}$ is a subfield of $K$, known as the \textbf{constant field}. The constant field of $K$ is algebraically closed in $K$. In the case of $T$-derivations, more is true:

\begin{lemma}
\label{constantelem}
Suppose that $(\cM,\delta)\models \Td$ and let $C$ be the constant field of $(\cM,\delta)$. Then $C$ is the underlying set of an elementary $L$-substructure of $\cM$.
\end{lemma}
\begin{proof}
Since $T$ has definable Skolem functions, it suffices to show that $C$ is $\dclL$-closed in $M$.
Given an $L(\0)$-definable function $f:M^n \to M$ be a tuple $\cv \in C^n$, we need to show that $f(\cv) \in C$. By passing to a subtuple, we may assume that $\cv$ is $\dclL(\0)$-independent, so $f$ is $\cC^1$ on some $L(\0)$-definable open neighborhood of $\cv$. Then $\delta f(\cv) = \Jac_f(\cv) \delta \cv =0$, so $f(\cv)\in C$.
\end{proof}

The following is a useful test to see if $\delta$ is a $T$-derivation:

\begin{lemma}\label{generic-point}
The following are equivalent:
\begin{enumerate}[(1)]
\item $(\cM,\delta)\models \Td$,
\item $\delta c = 0$ for all $c \in\dclL(\0)$ and $\delta f(\uv) = \Jac_f(\uv)\delta \uv$
for all $\dclL(\0)$-independent tuples $\uv$ and all $L(\0)$-definable functions $f$ which are $\cC^1$ in a neighborhood of $\uv$.
\end{enumerate}
\end{lemma}
\begin{proof}
Clearly (1) implies (2), as the constant function $x \mapsto c$ is $L(\0)$-definable for each $c \in \dclL(\0)$. Now suppose (2) holds, fix an $L(\0)$-definable $\cC^1$-function $f:U\to M$ with $U$ open and fix a tuple $\uv \in U$. If each component of $\uv$ is in $\dclL(\0)$, then 
$f(\uv) \in \dclL(\0)$, so $\delta f(\uv)= \Jac_f(\uv) \delta \uv = 0$ by (2). If there is some component of $\uv$ which is not in $\dclL(\0)$, then let $\uv'$ be a maximal $\dclL(\0)$-independent subtuple of $\uv$ and fix an $L(\0)$-definable map $g$ such that $g(\uv') = \uv$. As $\uv'$ is $\dclL(\0)$-independent, there is an open set $V$ containing $\uv'$ such that $g$ is $\cC^1$ on $V$ and such that $g(V) \subseteq U$. We have
\[
\delta f(\uv)\ =\ \delta (f\circ g)(\uv') \ =\ \Jac_{f\circ g}(\uv')\delta \uv'\ =\ \Jac_f\!\big(g(\uv')\big)\Jac_{g}(\uv')\delta \uv'\ =\ \Jac_f(\uv)\delta g(\uv')\ =\ \Jac_f(\uv)\delta \uv
\]
as required, where the second and fourth equality use (2) and the $\dclL(\0)$-independence of $\uv'$.
\end{proof}

By Lemma~\ref{generic-point}, the zero map (denoted by 0) is the only $T$-derivation on $\bP$. Thus, we have the following:

\begin{corollary}\label{primesubstructure}
$(\bP,0)$ is the prime substructure for $\Td$.
\end{corollary}

It is not true in general that any derivation on $\cM$ is a $T$-derivation, see Lemma~\ref{notalldelta}. However, this is true when $T= \RCF$. To prove this, we first need to establish two preservation results for compatibility.

\begin{lemma}
\label{basicchainrule} Let $U \subseteq M^n$, $V\subseteq M^m$ be $L(\0)$-definable, open sets. Let $f:U\to M$ be a $\cC^1$, $L(\0)$-definable function and let $g:V \to U$ be a $\cC^1$, $L(\0)$-definable map. If $\delta$ is compatible with $f$ and $g$, then $\delta$ is compatible with the composition $f \circ g$.
\end{lemma}
\begin{proof}
For $\uv \in V$, we have
\[
\delta f\big(g(\uv)\big)\ =\ \Jac_f\!\big(g(\uv)\big) \delta\big( g(\uv)\big)\ =\ \Jac_f\!\big(g(\uv)\big) \big(\Jac_{g}(\uv) \delta \uv\big)\ =\ \Jac_{f\circ g}(\uv) \delta \uv.\qedhere
\]
\end{proof}

\begin{lemma}
\label{implicitfunctions}
Let $f :V \to M^n$ be an $L(M)$-definable $\cC^1$-map on $V \subseteq M^{m+n}$ in variables $(\x,\y)$ and suppose that $\delta$ is compatible with $f$.
Let $g:U \to M^n$ be an $L(M)$-definable map on an open set $U\subseteq M^m$ such that $\Gamma(g) \subseteq V$.
Suppose that for all $\uv \in U$ we have $f\big(\uv,g(\uv)\big) = 0$ and that the determinant of the matrix $\frac{\partial f}{\partial \y}\big(\uv,g(\uv)\big)$ is nonzero. 
Then $g$ is $\cC^1$ and $\delta$ is compatible with $g$.
\end{lemma}
\begin{proof}
The map $g$ is $\cC^1$ by the implicit function theorem. Define the map $h:U \to K^n$ by $h(\x)= f(\x,g(\x))$, so $h(u)$ is identically zero on $U$. We have
\[
\frac{\partial f}{\partial \x}\big(\uv,g(\uv)\big) + \frac{\partial f}{\partial \y}\big(\uv,g(\uv)\big) \Jac_{g}(\uv)\ =\ \Jac_{h}(\uv)\ =\ \zero,
\]
thus,
\[
\frac{\partial f}{\partial \y}\big(\uv,g(\uv)\big)\Jac_{g}(\uv) \delta \uv\ =\ - \frac{\partial f}{\partial \x}\big(\uv,g(\uv)\big) \delta \uv.
\]
We also have
\[
 \frac{\partial f}{\partial \x}\big(\uv,g(\uv)\big)\delta \uv + \frac{\partial f}{\partial \y}\big(\uv,g(\uv)\big)\delta g(\uv)\ =\ \delta h(\uv)\ =\ \zero.
\]
We therefore have
\[
\frac{\partial f}{\partial \y}\big(\uv,g(\uv)\big)\delta g(\uv)\ =\ \frac{\partial f}{\partial \y}\big(\uv,g(\uv)\big)\Jac_{g}(\uv) \delta \uv.
\]
It remains to use the invertibility of $\frac{\partial f}{\partial \y}\big(\uv,g(\uv)\big)$.
\end{proof}

\begin{proposition}\label{semialg}
If $\delta$ is a derivation on $\cM$, then $\delta$ is an $\RCF$-derivation on $\cM$.
\end{proposition}
\begin{proof}
By quantifier elimination for $\RCF$ (in the language $L_{\ring}$ of ordered rings) and by Lemma~\ref{implicitfunctions}, it suffices to show that $\delta$ is compatible with every polynomial in $\Z[X]$. By repeated applications of Lemma~\ref{basicchainrule}, this amounts to showing that $\delta$ is compatible with addition, multiplication and the maps $x \mapsto nx$ for $n \in \Z$. These facts all follow readily from the definition of a derivative.
\end{proof}

The proof of Proposition~\ref{semialg} can often be adapted to check whether a derivation on $\cM$ is a $T$-derivation, at least when $T$ admits quantifier elimination in some natural language. If $T$ also has a universal axiomatization then checking whether or not a derivation is a $T$-derivation is even more simple since each $L(\0)$-definable function is given piecewise by $L$-terms. We give two examples below:

\begin{lemma}\label{Tanexp}\
\begin{enumerate}[(1)]
\item Let $\R_{\an}$ be the expansion of the real field by restricted analytic functions and let
$T_{\an}$ be its theory. Let $\cM \models T_{\an}$ and let $\delta$ be a derivation on $\cM$. Then $(\cM,\delta)\models \Td_{\an}$ if and only if $\delta$ is compatible with every restricted analytic function (restricted to the open unit disk).
\item Let $\R_{\anexp}$ be the expansion of $\R_{\an}$ by the total exponential function and let $T_{\anexp}$ be its theory. Let $\cM \models T_{\anexp}$ and let $\delta$ be a derivation on $\cM$. Then $(\cM,\delta)\models \Td_{\anexp}$ if and only if $\delta$ is compatible with every restricted analytic function and with the exponential function.
\end{enumerate}
Note that in (1) and (2) above, the model $\cM$ necessarily contains $\R$.
\end{lemma}
\begin{proof}
For (1), let $L_{\an}$ be the language of ordered fields extended by function symbols for
each restricted analytic function. Let $L_{\an}^*$ extend $L_{\an}$ by function symbols for multiplicative inversion and each $n$th root. 
As $\delta$ is a derivation, it is compatible with addition, multiplication, multiplicative inversion, and $n$th roots. Therefore, if $\delta$ is compatible with every restricted analytic function then $\delta$ is compatible with every $L_{\an}^*$-term by repeated applications of Lemma~\ref{basicchainrule}. By~\cite[Corollary 2.15]{DMM94}, each $L_{\an}(\0)$-definable function is given piecewise by $L_{\an}^*$-terms.

\medskip

For (2), let $L_{\anexp}$ be the language of ordered fields extended by function symbols for the exponential function and each restricted analytic function. Let $L_{\anexp}^*$ extend $L_{\anexp}$ by a function symbol for the logarithm function.
If $\delta$ is compatible with $\exp$, then
\[
\frac{\delta x}{x}\ =\ \frac{\delta \exp\big(\!\log(x)\big)}{x}\ =\ \frac{\exp\big(\!\log(x)\big)\delta \log(x)}{x}\ =\ \delta \log(x)
\]
for all $x>0$, so $\delta$ is also compatible with $\log$. By~\cite[Corollary 4.7]{DMM94}, each $L_{\anexp}(\0)$-definable function is given piecewise by $L_{\anexp}^*$-terms. By the same reasoning as above, if $\delta$ is compatible with every restricted analytic function and with the exponential function then $\delta$ is a $T_{\anexp}$-derivation.
\end{proof}

\begin{lemma}\label{notalldelta}
There is an o-minimal theory $T \supsetneq \RCF$, a model $\cM \models T$ and a derivation $\delta$ on $\cM$ such that $\delta$ is not a $T$-derivation.
\end{lemma}
\begin{proof}
Let $\xi\in \R$ be transcendental over $\Q$ and let $M$ be the smallest real-closed subfield of $\R$ containing $\xi$. Set $L= L_{\ring}\cup \{\xi\}$ where $\xi$ is a new constant symbol, so $M$ admits a natural expansion to an $L$-structure $\cM$ where $\xi$ is interpreted in the obvious way. Let $T$ be the complete $L$-theory of $\cM$, so $T$ is o-minimal. Basic facts about derivations tell us that there is a derivation $\delta$ on $\cM$ with $\delta \xi = 1$. However $\xi\in\dclL(\0)$, so if $\delta$ were a $T$-derivation, we would have $\delta \xi = 0$ by Lemma~\ref{generic-point}.
\end{proof}

Lemmas~\ref{Tanexp} and~\ref{notalldelta} raise a natural question (suggested by the anonymous referee), which is open at this time:

\begin{question}\label{isitaTexp}
Let $\R_{\exp}$ be the expansion of the real field by the total exponential function and let $T_{\exp}$ be its theory. Let $\cM \models T_{\exp}$ and let $\delta$ be a derivation on $\cM$ which is compatible with the exponential function. Is $\delta$ necessarily a $T_{\exp}$-derivation on $\cM$?
\end{question}

An anwer to Question~\ref{isitaTexp} likely requires an understanding of the definable closure in models of $T_{\exp}$.

\begin{lemma}\label{basicTderivation}
Suppose that $(\cM,\delta)\models \Td$. Let $k>0$ and let $f$ be an $L(M)$-definable $\cC^k$-function on an open set $U\subseteq M^n$. Then there is a unique $L(M)$-definable $\cC^{k-1}$-function $f^{[\delta]}:U \to M$ such that
\[
\delta f(\uv)\ =\ f^{[\delta]}(\uv)+\Jac_f(\uv) \delta \uv
\]
for all $\uv \in U$. Moreover, if $f$ is $L(A)$-definable where $A \subseteq \ker(\delta)$, then $f^{[\delta]} = 0$
\end{lemma}
\begin{proof}
If such a function $f^{[\delta]}$ exists, then it is uniquely determined by $f^{[\delta]}(\uv) = \delta f(\uv) -\Jac_f(\uv) \delta \uv$. We show existence:
fix $\av \in M^m$, an $L(\0)$-definable function $F$ such that $f = F_\av$ and an $L(\0)$-definable set $W$ such that $W_{\av} = U$. By replacing $W$ with the subset
\[
\big\{\x \in W: W_\x\text{ is open and }F_\x\text{ is }\cC^k\text{ on }W_\x\text{ for all }\x \in \pi_m(W)\big\},
\]
we may assume that $F$ and $W$ satisfy the hypotheses of Corollary~\ref{Ckfiber}. Thus, we can conclude that $\av$ is contained in an $L(\0)$-definable $\cC^k$-cell $D \subseteq M^m$ such that $F|_{W \cap (D\times M^n)}$ is $\cC^k$. By replacing $W$ with $W \cap (D\times M^n)$, we may assume that $F$ is $\cC^k$ on $W$. Take an open set $\tilde{W} \supseteq W$ and a $\cC^k$-function $G = G(\x,\y):\tilde{W} \to M$ such that $G|_W = F|_W$, so $G_\av(\uv) = f(\uv)$ for all $\uv \in U$. Since $G$ is $L(\0)$-definable, we have for $\uv \in U$ that
\[
\delta f(\uv)\ =\ \delta G(\av,\uv)\ =\ \Jac_G(\av,\uv) (\delta \av,\delta \uv)\ =\ f^{[\delta]}(\uv)+\Jac_f(\uv) \delta \uv,
\]
where $f^{[\delta]}(\y ):= \frac{\partial G}{\partial \x}(\av,\y) \delta \av$. Clearly, $f^{[\delta]}$ is a $\cC^{k-1}$-function and if $\delta \av = 0$ then $f^{[\delta]} = 0$.
\end{proof}

As is the case for derivations, one can extend a $T$-derivation $\delta$ on $\cM$ to $\cN \succeq_L \cM$ by specifying the values of $\delta$ on a $\dclL$-basis for $\cN$ over $\cM$.

\begin{lemma}
\label{transext}
Suppose that $(\cM,\delta)\models \Td$.
Let $\cN \succeq_L \cM$ and suppose that $\cN = \cM\langle A\rangle$ where $A$ is a $\dclL(M)$-independent subset of $N$.
Then for any map $s:A \to N$, there is a unique extension of $\delta$ to a $T$-derivation on $\cN$ (also denoted by $\delta$) such that $\delta a= s(a)$ for all $a \in A$.
\end{lemma}
\begin{proof}
We need to determine the value of $\delta f(\av)$ for each $n$, each $L(M)$-definable function $f:N^n\to N$ and each $n$-tuple $\av$ of distinct elements from $A$. Since $\av$ is a $\dclL(M)$-independent tuple, $f$ is $\cC^1$ on some open $L(M)$-definable open neighborhood $U$ of $\av$.
By Lemma~\ref{basicTderivation}, there is an $L(M)$-definable function $f^{[\delta]}: U \to M$ such that $\delta f(\uv) = f^{[\delta]}(\uv)+\Jac_f(\uv) \delta \uv$ for all $\uv \in U^\cM$.
For $\delta$ to satisfy the uniqueness condition in this lemma, we only have one choice, so we set
\[
\delta f(\av)\ :=\ f^{[\delta]}(\av)+\Jac_f(\av) s(\av).
\]
Clearly, $\delta a = s(a)$ for each $a \in A$. This assignment is also well-defined, for if $f$ and $g$ are $L(M)$-definable functions such that $f(\av) = g(\av)$ for some tuple $\av$ of distinct elements from $A$, then by $L(M)$-independence of $\av$, there is some $L(M)$-definable open neighborhood of $\av$ on which $f= g$ and thus, on which $f^{[\delta]} = g^{[\delta]}$ and $\Jac_f = \Jac_g$.

\medskip

We claim that $\delta$ is a $T$-derivation. We need to show that $\delta g(\uv) = \Jac_g(\uv) \delta u$ for all $\uv \in N^m$ and all $L(\0)$-definable $\cC^1$-functions $g:U \to N$ where $U$ is an open neighborhood of $\uv$.
Take an $n$-tuple $\av$ of distinct elements from $A$ and an $L(M)$-definable map $f:N^n \to N^m$ such that $\uv := f(\av)$ and take an open $L(M)$-definable set $V$ containing $\av$ such that $f$ is $\cC^1$ on $V$ and such that $f(V)\subseteq U$. Set $h:= g \circ f$ and set $\sv:=s(\av)$.
By definition, we have
\begin{equation}\label{transexteq}
\delta g(\uv)\ =\ \delta h(\av)\ =\ h^{[\delta]}(\av)+\Jac_h(\av) \sv, \qquad \delta \uv\ =\ \delta f(\av)\ =\ f^{[\delta]}(\av)+\Jac_{f}(\av) \sv.
\end{equation}
For all $\x \in V^\cM$, we have that
\[
h^{[\delta]}(\x)+\Jac_h(\x) \delta \x\ =\ \delta h(\x)\ =\ \delta g(f(\x))\ =\ \Jac_g\!\big(f(\x)\big)\big(f^{[\delta]}(\x)+\Jac_{f}(\x) \delta \x\big).
\]
Using also that $\Jac_h(\x) = \Jac_{g \circ f}(\x) = \Jac_g\!\big(f(\x)\big)\Jac_{f}(\x)$, we see that $h^{[\delta]}(\x) = \Jac_g\!\big(f(\x)\big)f^{[\delta]}(\x)$ for all $\x \in V^\cM$. As every object we've considered since (\ref{transexteq}) has been $L(M)$-definable, we have by elementarity that
\[
h^{[\delta]}(\av)\ =\ \Jac_g\!\big(f(\av)\big)f^{[\delta]}(\av)\ =\ \Jac_g(\uv)f^{[\delta]}(\av),\qquad \Jac_h(\av)\ =\ \Jac_g\!\big(f(\av)\big)\Jac_{f}(\av)\ =\ \Jac_g(\uv)\Jac_{f}(\av).
\]
This along with the identities in (\ref{transexteq}) gives that
\[
\delta g(\uv)\ =\ \Jac_g(\uv)f^{[\delta]}(\av)+\Jac_g(\uv)\Jac_{f}(\av) \sv\ =\ \Jac_g(\uv) \delta \uv. \qedhere
\]
\end{proof}

\subsection{Examples of $T$-derivations} Given any $\cM \models T$, the map $\delta: M \to M$ which takes constant value 0 is a $T$-derivation. In this subsection, we explore some nontrivial $T$-derivations.

\begin{example}[Power series]
Let $\R((t^\Q))$ be the field of formal power series with coefficients in
$\R$, exponents in $\Q$ and well-ordered support.
By~\cite[Corollary 2.1]{DMM94}, $\R((t^\Q))$ admits a canonical expansion to a model of $T_{\an}$ where restricted analytic functions are defined via Taylor series expansion. Let $\frac{\dif}{\dif t}$ be the usual formal derivation on $\R((t^\Q))$. Since $\frac{\dif}{\dif t}$ commutes with infinite sums, it is compatible with restricted analytic functions. By Lemma~\ref{Tanexp}, $\frac{\dif}{\dif t}$ is a $T_{\an}$-derivation on $\R((t^\Q))$.
\end{example}

\begin{example}[Surreal numbers]
Let $\mathbf{No}$ be the class-sized ordered field of surreal numbers. Then $\mathbf{No}$ admits a canonical expansion to a model of $T_{\anexp}$; see~\cite[Theorem 2.1]{DE01}. Berarducci and Mantova have defined a derivation $D$ on the surreal numbers which is compatible with the exponential function and commutes with infinite sums~\cite[Theorem 6.30]{BM18}. Thus, $D$ is a $T_{\anexp}$-derivation, again by Lemma~\ref{Tanexp}.
\end{example}

\begin{example}[Transseries]
Let $\mathbb{T}$ be the ordered field of logarithmic-exponential transseries and let $\frac{\dif}{\dif x}$ be the usual derivation on $\mathbb{T}$ (see~\cite{ADH17} for a detailed definition). By~\cite[Corollary 2.8]{DMM97}, $\mathbb{T}$ admits a canonical expansion to a model of $T_{\anexp}$ and by Lemma~\ref{Tanexp}, $\frac{\dif}{\dif x}$ is a $T_{\anexp}$-derivation on $\mathbb{T}$.
\end{example}

\begin{example}[Hardy fields]
Let $\overline{\R}$ be an arbitrary o-minimal expansion of the real field in a language $L$. Let $T$ be the elementary $L$-theory of $\overline{\R}$. We define an equivalence relation on $L(\R)$-definable functions $f,g:\R\to \R$ by setting $f \sim g$ if there is some $a \in \R$ such that $f|_{(a,\infty)} = g|_{(a,\infty)}$. If $f \sim g$, we say that $f$ and $g$ have the same \emph{germ at infinity}. We let $[f]_\sim$ be the equivalence class of $f$ and we set
\[
M\ :=\ \big\{[f]_{\sim}:f:\R\to \R \text{ is }L(\R)\text{-definable}\big\}.
\]
There is a natural expansion of $M$ to an $L$-structure $\cM$. For example, if $R$ is an $n$-ary relation symbol in $L$, we interpret $R$ in $\cM$ by
\[
\cM\models R\big([f_1]_\sim,\ldots,[f_n]_\sim\big)\ :\Longleftrightarrow\ \overline{\R} \models R\big(f_1(x),\ldots,f_n(x)\big) \text{ for all sufficiently large }x
\]
where this is well-defined by o-minimality. Under this expansion, we have $\overline{\R} \preceq_L\cM$, where $\overline{\R}$ is identified with the germs of constant functions, see~\cite{DMM94}. Now we define $\delta:M\to M$ by setting
\[
\delta[f]_\sim\ :=\ [f']_\sim.
\]
We note that above, $f$ may not be everywhere differentiable, but it is differentiable at all sufficiently large $x$ so it makes sense to talk about the germ of $f'$. Then $(\cM,\delta)$ is a Hardy field and it is easy to check that $\delta$ is a $T$-derivation on $\cM$ (just use the chain rule from elementary calculus).
\end{example}

\begin{remark}
Let $(\cM,\delta)$ be as in any of the four examples above. Then the constant field of $(\cM,\delta)$ is $\R$ and the pair $(\cM,\R)$ is a \textbf{tame pair}, as defined in~\cite[Section 5]{DL95}. In~\cite[Section 2]{BKW10}, the authors construct derivations on the real exponential field.
These derivations can be taken to be $T_{\exp}$-derivations, but we are not sure if these derivations are \emph{necessarily} $T_{\exp}$-derivations, see Question~\ref{isitaTexp}. Also relevant to Question~\ref{isitaTexp} are the ``$E$-derivations'' used by Kirby to study exponential algebraicity in~\cite{Ki10}.
\end{remark}

\subsection{The Lie algebra of $T$-derivations}\label{subsec:Lie}
Let $\Der_T(\cM)$ be the set of $T$-derivations on $\cM$. Given $\delta,\epsilon \in \Der_T(\cM)$ and $a_1,a_2 \in M$, one can easily check that $a_1\delta+a_2\epsilon \in \Der_T(\cM)$, so $\Der_T(\cM)$ naturally has the structure of an $M$-vector space. In this subsection, we show that it has the structure of a Lie algebra. We define a Lie bracket on $\Der_T(\cM)$ by 
\[
[\delta, \epsilon]\ :=\ \delta\epsilon - \epsilon\delta
\]
(where $\delta\epsilon$ is the composition of $\delta$ with~$\epsilon$).
It is routine to verify that this operation formally satisfies the Lie bracket axioms, so we only have to ensure that $[\delta, \epsilon]$ is indeed a $T$-derivation.

\begin{lemma}\label{twoderivs}
If $\delta,\epsilon \in \Der_T(\cM)$ then $[\delta,\epsilon]\in\Der_T(\cM)$.
\end{lemma}
\begin{proof}
Set $\gamma := [\delta,\epsilon]$. By Lemma~\ref{generic-point}, we only need to check that $\gamma(c) = 0$ for all $c \in \dclL(\0)$ and that $\gamma f(\uv) = \Jac_f(\uv)\gamma \uv$
for all $\dclL(\0)$-independent tuples $\uv\in M^n$ and all $L(\0)$-definable functions $f$ which are $\cC^1$ in a neighborhood of $\uv$. The fact that $\gamma c = 0$ for all $c \in \dclL(\0)$ follows from the fact that $\delta c= \epsilon c = 0$. Fix $f$ and a $\dclL(\0)$-independent tuple $\uv$. By $\dclL$-independence, we may assume that $f = f(\y)$ is $\cC^2$ in an open neighborhood $U$ of $\uv$.
We have
\[
\delta\epsilon f(\uv)\ =\ \delta\big(\Jac_f(\uv) \epsilon \uv\big)\ =\ \sum_{j=1}^n\delta\left(\frac{\partial f}{\partial y_j}(\uv)\right)\epsilon u_j+\Jac_f(\uv) \delta\epsilon \uv\ =\ \sum_{i,j=1}^n \frac{\partial^2 f}{\partial y_i\partial y_j}(\uv)\delta u_i\epsilon u_j+\Jac_f(\uv) \delta\epsilon \uv.
\]
Likewise, we have
\[
\epsilon\delta f(\uv)\ =\ \sum_{i,j=1}^n \frac{\partial^2 f}{\partial y_i\partial y_j}(\uv)\epsilon u_i\delta u_j+\Jac_f(\uv) \epsilon\delta \uv.
\]
so by symmetry of second derivatives, we have
\[
\gamma f(\uv)\ =\ \delta\epsilon f(\uv)-\epsilon\delta f(\uv) \ =\ \Jac_f(\uv) \delta\epsilon \uv-\Jac_f(\uv) \epsilon\delta \uv\ =\ \Jac_f(\uv) \gamma \uv. \qedhere
\]
\end{proof}

Let $K$ be a subfield of $M$. We say that $\delta$ is a \textbf{$T$-derivation over $K$} if $\delta(c) = 0$ for all $c \in K$. 

\begin{lemma}
The set of $T$-derivations over $K$ is a Lie sub-algebra of $\Der_T(\cM)$.
\end{lemma}
\begin{proof}
Set $L(K)$ extend $L$ by constant symbols for each $c\in K$ and let $T_K$ be the complete $L(K)$-theory of $\cM$. Clearly, any $T_K$-derivation on $\cM$ is a $T$-derivation over $K$. Conversely, if $\delta$ is a $T$-derivation over $K$ then $\delta$ is a $T_K$-derivation by the ``moreover'' part of Lemma~\ref{basicTderivation}. Thus, the set of $T$-derivations over $K$ is exactly the $\Der_{T_K}(\cM)$, which is a Lie sub-algebra of $\Der_T(\cM)$.
\end{proof}

\section{The \texorpdfstring{$\delta$}{δ}-closure operator} \label{sec:dims}
Let $(\cM,\delta) \models \Td$. In this section, we develop a $\delta$-closure operator on
$M$. First, some notation: given $a \in M$, we define the \textbf{jets} of $a$:
\[
\jet^n(a)\ :=\ (a,\delta a,\ldots, \delta^n a),\qquad \jet^\infty(a) := (\delta^ia)_{i < \omega}.
\]
Given $\av \in M^m$, $B \subseteq M^m$ and $\alpha \in \N \cup \{\infty\}$, we set
\[
\jet^\alpha(\av)\ :=\ \big(\jet^\alpha(a_1),\ldots,\jet^\alpha(a_m)\big),\qquad \jet^\alpha(B)\ :=\ \big\{\jet^\alpha(\bv):\bv \in B\big\}.
\]
For simplicity of notation, we let $\jet^{-1}(\av)$ be the empty tuple and we let $\jet^{-1}(B)$ be the empty set. 
\begin{definition}\label{delclosure}
Given $a \in M$ and $B \subseteq M$, we say that \textbf{$a$ is in the $\delta$-closure of $B$}, written
$a \in \cld(B)$, if 
\[
\rkL\big(\jet^\infty(a)|\jet^\infty(B)\big)\ <\ \aleph_0.
\]
\end{definition}

This section is devoted to showing that $(\cM,\cld)$ is a finitary matroid and exploring the corresponding rank function.

\subsection{Quasi-endomorphisms}\label{subsec:quasi-endo}
In this subsection, we fix a set $X$ and a closure operator $\cl:\cP(X) \to \cP(X)$ such that $(X,\cl)$ is a finitary matroid. Let $\rk$ denote the associated cardinal-valued rank function. We say that a map $\delta:X \to X$ is a \textbf{quasi-endomorphism of $(X,\cl)$} if
\[
\rk(\delta A | AB\delta B)\ \leq\ \rk(A|B)
\]
for all $A,B \subseteq X$. Fix a quasi-endomorphism $\delta$. Throughout this subsection, $A$, $B$ and $C$ denote subsets of $X$ and $a$, $b$ and $c$ denote elements of $X$. We continue to use the $\jet^n$ and $\jet^\infty$ notation introduced in the beginning of this section. Though we are working with an abstract finitary matroid, the example to keep in mind is of course the case where $(X,\cl)= (\cM,\dclL)$ and where $\delta$ is a $T$-derivation on $\cM$.

\medskip

We define $\cld: \cP (X) \to \cP (X)$ as in Definition~\ref{delclosure}:
\[
a \in \cld(B)\ :\Longleftrightarrow\ \rk\big(\jet^\infty (a) | \jet^\infty (B)\big) < \aleph_0.
\]
Note that $\cl(B) \subseteq \cld(B)$ and that $\jet^\infty(B) \subseteq \cld(B)$.

\begin{lemma}\label{dcleq}
The following are equivalent:
\begin{enumerate}[(1)]
\item $a \in \cld(B)$,
\item $\rk\big(\jet^n(a)|\jet^\infty(B)\big) \leq n$ for some $n$,
\item $\delta^{n}a \in \cl\!\big(\jet^{n-1}(a)\jet^\infty(B)\big)$ for some $n$
\item there are $n$ and $m$ such that $\delta^{k}a \in \cl\!\big(\jet^{n-1}(a)\jet^{m+k}(B)\big)$ for all $k \geq n$.
\end{enumerate}
\end{lemma}
\begin{proof}
Suppose that (1) holds and set $n:= \rk\big(\jet^\infty(a)|\jet^\infty(B)\big)$. Then $\rk\big(\jet^n(a)|\jet^\infty(B)\big) \leq n$.

\medskip

Now suppose that (2) holds and let $n$ be least such that $\rk\big(\jet^n(a)|\jet^\infty(B)\big) \leq n$. We have that 
\[
\rk\big(\jet^{n-1}(a)|\jet^\infty(B)\big)\ =\ n
\]
 by minimality of $n$. Thus, $\delta^na\in \cl\!\big(\jet^{n-1}(a)|\jet^\infty(B)\big)$. 

\medskip

Suppose that (3) holds. As $\cl$ is finitary, there is some $m$ such that $\delta^{n}a \in \cl\!\big(\jet^{n-1}(a)\jet^{m+n}(B)\big)$. Set $B':=\jet^{n-1}(a)\jet^{m+n}(B)$, so $\delta^na \in \cl(B')$. Since $\delta$ is a quasi-endomorphism, we have that 
\[
\rk(\delta^{n+1}a | \delta^naB'\delta B')\leq \rk(\delta^na|B')\ =\ 0
\]
so $\delta^{n+1}a \in \cl( \delta^naB'\delta B')$. Since $\delta^na \in \cl(B')$, we have that $\delta^naB'\delta B' \subseteq \cl(\jet^{n-1}(a)\jet^{m+n+1}(B))$, so $\delta^{n+1} \in \cl(\jet^{n-1}(a)\jet^{m+n+1}(B))$. By induction, we have that $\delta^ka\in \cl(\jet^{n-1}(a)\jet^{m+k}(B))$ for all $k \geq n$.

\medskip 

The final implication, (4) implies (1), is clear.
\end{proof}

We will use the following fact frequently, often without mentioning it. It follows from (3) of Lemma~\ref{dcleq}.
\begin{fact}\label{notincl}
$a \not\in \cld(B)$ if and only if $\jet^\infty(a)$ is $\cl\big(\jet^\infty(B)\big)$-independent.
\end{fact}

\begin{proposition}\label{ispregeom}
$(X,\cld)$ is a finitary matroid. 
\end{proposition}
\begin{proof}
It is clear that if $A \subseteq B$ then $A \subseteq \cld(A) \subseteq \cld(B)$. The fact that $\cld$ is finitary follows from (3) of Lemma~\ref{dcleq} and the fact that $\cl$ is finitary. We will show that $\cld\!\big(\!\cld(B)\big) = \cld(B)$. Fix $a \in \cld\!\big(\!\cld(B)\big)$ and fix a finite set $C \subseteq \cld(B)$ such that $a \in \cld(C)$. Then 
\[
\rk\big(\jet^\infty(a)|\jet^\infty(B)\big)\ \leq\ \rk\big(\jet^\infty a|\jet^\infty(C)\big) + \rk\big(\jet^\infty (C)|\jet^\infty(B)\big).
\]
Since $C$ is finite, both summands on the right side of the above inequality are finite.

\medskip

It remains to show that $\cld$ satisfies the exchange property. Fix $a$, $b$ and $B$ such that $a \in \cld(Bb) \setminus \cld(B)$. By (2) of Lemma~\ref{dcleq}, there is a natural number $n$ such that $\rk\big(\jet^n(a)|\jet^\infty(Bb)\big) \leq n$. Since $\cl$ is finitary, we may find a natural number $m$ such that $\rk\big(\jet^n(a)|\jet^\infty(B)\jet^m(b)\big) \leq n$. We have
\begin{equation}\label{matroideq1}
 \rk\big(\jet^m(b)\jet^n(a)|\jet^\infty(B)\big)\ =\ \rk\big(\jet^n(a)|\jet^m(b)\jet^\infty(B)\big) + \rk\big(\jet^m(b)|\jet^\infty(B)\big)\ \leq\ n+m+1.
\end{equation}
On the other hand,
\[
\rk\big(\jet^m(b)\jet^n(a)|\jet^\infty(B)\big)\ =\ \rk\big(\jet^m(b)|\jet^n(a)\jet^\infty(B)\big) + \rk\big(\jet^n(a)|\jet^\infty(B)\big)
\]
Since $a \not\in \cld(B)$, we have $\rk\big(\jet^n(a)|\jet^\infty(B)\big) = n+1$. This gives us
\begin{equation}\label{matroideq2}
 \rk\big(\jet^m(b)\jet^n(a)|\jet^\infty(B)\big)\ =\ \rk\big(\jet^m(b)|\jet^n(a)\jet^\infty(B)\big)+n+1.
\end{equation}
Combining (\ref{matroideq1}) and (\ref{matroideq2}), we get
\[
\rk\big(\jet^m(b)|\jet^n(a)\jet^\infty(B)\big)\ \leq\ m,
\]
so $b \in \cld(Ba)$, again by (2) of Lemma~\ref{dcleq}.
\end{proof}

As $(X,\cld)$ is a finitary matroid, it has an associated rank function which we call the $\delta$-rank and which we denote by $\rkd$.
The next proposition gives a method of computing the $\delta$-rank of finite sets:

\begin{proposition}\label{rkcompute}
Let $A$ be finite and suppose that $\delta B \subseteq B$. Then 
\[
\rkd(A|B)\ =\ \lim\limits_{k\to \infty}\frac{\rk\big(\jet^k(A)|B\big)}{k+1} .
\]
In particular, this limit exists.
\end{proposition}
\begin{proof}
Given a finite set $A$ and an element $a$, set
\[
r(A)\ :=\ \lim\limits_{k\to \infty}\frac{\rk\big(\jet^k(A)|B\big)}{k+1},\qquad r(a|A)\ :=\ \lim\limits_{k\to \infty}\frac{\rk\big(\jet^k(a)|\jet^k(A)B\big)}{k+1}
\]
 (assuming that these limits exist). We prove this proposition by induction on $|A|$. Clearly, $\rkd(\0|B) = r(\0) = 0$. Fix $A$ and suppose that $r(A) = \rkd(A|B)$. We want to show that $r(Aa) = \rkd(Aa|B)$ for some arbitrary $a \in X\setminus A$. Note that $r(Aa) = r(A) +r(a|A) = \rkd(A|B)+r(a|A)$ 
by our induction hypothesis, so it suffices to show that $r(a|A) = \rkd(a|AB)$.

\medskip

If $a \not\in \cld(AB)$, then $\rk\big(\jet^k(a)|\jet^k(A)B\big) =k+1$ for each $k$ by (2) of Lemma~\ref{dcleq}, so $r(a|A) = 1 = \rkd(a|AB)$. 
Suppose that $a \in \cld(AB)$. By (4) of Lemma~\ref{dcleq}, there are $m$ and $n$ such that $\delta^{k}a \in \cl\!\big(\jet^{n-1}(a)\jet^{m+k}(A)B\big)$ for all $k \geq n$. For these $k$, we have
\[
\rk\big(\jet^{m+k}(a)|\jet^{m+k}(A)B\big)\ \leq\ \rk\big(\jet^{n-1}(a)|\jet^{m+k}(A)B\big) + \rk\big(\{\delta^{k+1}a,\ldots,\delta^{k+m}a\}|\jet^{m+k}(A)B\big)\ \leq\ n+m.
\]
Therefore, we have that
\[
r(a|A)\ =\ \lim\limits_{k \to \infty}\frac{\rk\big(\jet^{m+k}(a)|\jet^{m+k}(A)B\big)}{m+k+1}\ \leq\ \lim\limits_{k \to \infty}\frac{n+m}{m+k+1}\ =\ 0\ =\ \rkd(a|AB).
\]
It remains to note that $r(a|A)\geq 0$.
\end{proof}
 
\begin{corollary}
If $A$ is finite and $\delta B \subseteq B$, then 
\[
\rkd(A|B)\ =\ \lim\limits_{k\rightarrow\infty}\rk\big(\delta^{k}A|\jet^{k-1}(A)B\big) .
\]
\end{corollary}
\begin{proof}
Note that $\rk\big(\jet^k(A)gB\big) = \sum_{n=0}^k\rk\big(\delta^n A|\jet^{n-1}(A)B\big)$.
Since $\delta$ is a quasi-endomorphism, we have that $\rk\big(\delta^{n+1} A|\jet^{n}(A)B\big) \leq \rk\big(\delta^n A|\jet^{n-1}(A)B\big)$. This means that the map 
\[
n\ \mapsto\ \rk\big(\delta^{n+1} A|\jet^{n}(A)B\big):\N\ \to\ \N
\]
is decreasing, so it is eventually constant. From this, it easily follows that
\[
 \lim\limits_{k\to \infty}\frac{\rk\big(\jet^k(A)|B\big)}{k+1}\ =\ \lim\limits_{k\to \infty}\frac{\sum_{n=0}^k\rk\big(\delta^n A|\jet^{n-1}(A)B\big)}{k+1}\ =\ \lim\limits_{k\rightarrow\infty}\rk\big(\delta^{k}A|\jet^{k-1}(A)B\big) .
\]
The result then follows from Proposition~\ref{rkcompute}.
\end{proof}

\subsection{The $\delta$-closure in models of $\Td$}
In order to apply the results to the previous subsection to $(\cM,\delta) \models \Td$, we need to show the following:

\begin{lemma}\label{isquasi}
If $(\cM,\delta) \models \Td$, then $\delta$ is a quasi-endomorphism of $(\cM,\dclL)$.
\end{lemma}
\begin{proof}
 Fix $A,B\subseteq M$ and let $A' \subseteq A$ be a $\dclL(B)$-independent set such that $A \subseteq \dclL(A'B)$. If we can show that $\delta A \subseteq \dclL(A'B\delta A'\delta B)$, then we would have that 
 \[
\rkL(\delta A|AB\delta B)\ \leq\ |\delta A'|\ \leq\ |A'| \ =\ \rkL(A|B).
\]
Thus, by replacing $B$ with $A'B$, we assume that $A \subseteq \dclL(B)$ and we will show that $\delta A \subseteq \dclL(B\delta B)$.
Given $a \in A$, we may write $a = f(\bv)$ for some $\bv$ from $B$ and some $L(\0)$-definable function $f$. By passing to a subtuple, we may assume that $\bv$ is $\dclL(\0)$-independent. Therefore, there is an open, $L(\0)$-definable set $U$ such that $\bv \in U$ and such that $f|_U$ is $\cC^1$. Then we have $\delta a= \Jac_f(\bv)\delta\bv$, so $\delta a\in \dclL(B\delta B)$.
\end{proof}

We summarize the results from the previous subsection in this context below:

\begin{corollary}\label{practicalrk}
If $(\cM,\delta) \models \Td$, then $(\cM,\cld)$ is a finitary matroid and for any finite set $A \subseteq M$ and any $B \subseteq M$ with $\delta B \subseteq B$, we have
\[
\rkd(A|B)\ =\ \lim\limits_{k\to \infty}\frac{\rkL\big(\jet^k(A)|B\big)}{k+1}\ =\ \lim\limits_{k\rightarrow\infty}\rkL\big(\delta^{k}A|\jet^{k-1}(A)B\big) .
\]
where $\rkd$ is the rank function corresponding to $\cld$.
\end{corollary}

\section{Generic \texorpdfstring{$T$}{T}-derivations} \label{sec:TdG}
In this section, we show that $\Td$ has a model completion and we study the properties of this model completion. For the remainder of this section, we fix a model $(\cM,\delta) \models \Td$.

\begin{definition}
We say that the $T$-derivation $\delta$ is \textbf{generic} if for every $n$ and every $L(M)$-definable set $A \subseteq M^{n+1}$, if $\dimL\big(\pi_n(A) \big) = n$ then there is some $a \in M$ such that $\jet^n(a) \in A$. Let $\TdG$ be the $\Ld$-theory extending $\Td$ by the axiom scheme which asserts that $\delta$ is generic.
\end{definition}

\subsection{Expanding models of $T$ and $\Td$ to models of $\TdG$}
In this subsection, we show that $(\cM,\delta)$ extends to a model of $\TdG$. We also investigate which models of $T$ admit an expansion to a model of $\TdG$. 

\begin{lemma}\label{Tderextensions2}
Let $\cN$ be an elementary extension of $\cM$ with $\rkL(N|M)= n$. Let $A \subseteq N^{n+1}$ be an $L(M)$-definable set with $\dimL\big(\pi_n(A) \big) = n$. Then there is $b \in N$ and an extension of $\delta$ to a $T$-derivation on $\cN$ such that $\jet^n(b) \in A$.
\end{lemma}
\begin{proof}
We claim that there is some $\dclL(M)$-independent tuple $\av \in N^n$ such that $\av \in \pi_n(A)$. We construct $\av$ coordinate by coordinate. Fix $i\in \{1,\ldots, n\}$ and suppose we have already chosen a $\dclL(M)$-independent tuple $\av'=(a_1,\ldots,a_{i-1}) \in \pi_{i-1}(A)$. We need to find $a_i \in N \setminus \dclL(M\av')$ with $(\av',a_i) \in \pi_i(A)$. We have that $\pi_i(A)_{\av'}$ is an open interval with endpoints $r_1<r_2\in \dclL(M\av')$. Take $b \in N \setminus \dclL(M\av')$ with $b>0$. Set
\[
a_i\ :=\ r_1+\frac{1}{\frac{1}{r_2-r_1}+b}.
\]
Then $a_i\not\in \dclL(M\av')$ and $r_1<a_i<r_2$, as required.

\medskip

With the claim proven, we may assume that $\cN = \cM\langle \av\rangle$ for some $\av \in N^n$ with $\av \in \pi_n(A)$. By definable choice, there is an $L(M)$-definable map $f:\pi_n(A) \to N$ such that $\Gamma(f) \subseteq A$.
By Lemma~\ref{transext}, there is a unique extension of $\delta$ to a $T$-derivation on $\cM\langle \av\rangle$ such that $\delta a_i = a_{i+1}$ for $i \in\{ 1,\ldots,n-1\}$ and such that $\delta a_n = f(\av)$. Then $\jet^n(a_1) \in A$.
\end{proof}

\begin{proposition}\label{extending}
Let $\cN$ be an elementary extension of $\cM$ and suppose that $\rkL(N|M) = |N| \geq |T|$. Then there is an extension of $\delta$ to a $T$-derivation on $\cN$ such that $(\cN,\delta) \models \TdG$.
\end{proposition}
\begin{proof}
Set $\kappa := |N|$ and take a $\dclL(M)$-independent set $B\subseteq N$ such that $\cN = \cM\langle B\rangle$. Then $|B| = \kappa$ by assumption. Let $B_1,B_2,\ldots$ be disjoint subsets of $B$ of cardinality $\kappa$ such that $\bigcup_{k>0} B_k = B$.
We will construct $\Ld$-structures $(\cN_k,\delta)\models \Td$ such that
\begin{itemize}
\item $(\cN_0,\delta) = (\cM,\delta)$ and $\cN_{k+1} = \cN_k\langle B_{k+1}\rangle$ for $k \geq 0$;
\item $(\cN_k,\delta) \subseteq (\cN_{k+1},\delta)$ and $\bigcup_k (\cN_k,\delta) \models \TdG$.
\end{itemize}
Suppose that we have already constructed $(\cN_k,\delta)$. Let $\big((A_\rho,n_\rho)\big)_{\rho<\kappa}$ be an enumeration of all pairs $(A,n)$ such that $A\subseteq N_k^{n+1}$ is an $L(N_k)$-definable set with $\dimL\big(\pi_n(A)\big) = n$. Let $(B_\rho)_{\rho<\kappa}$ be an enumeration of pairwise disjoint finite subsets of $B_{k+1}$ such that $|B_\rho| = n_\rho$ and $\bigcup_{\rho<\kappa}B_\rho = B_{k+1}$. We define $\big((\cN_{k,\rho},\delta)\big)_{\rho<\kappa}$ as follows:
\begin{itemize}
\item set $(\cN_{k,0},\delta) := (\cN_k,\delta)$;
\item if $\rho$ is a limit ordinal, set $(\cN_{k,\rho},\delta) := \bigcup_{\gamma<\rho}(\cN_{k,\gamma},\delta)$;
\item set $\cN_{k,\rho+1}:= \cN_{k,\rho}\langle B_\rho\rangle$ and use Lemma~\ref{Tderextensions2} to extend $\delta$ to a $T$-derivation on $\cN_{k,\rho+1}$ such $\jet^{n_\rho}(b) \in A_\rho$ for some $b \in N_{k,\rho+1}$.
\end{itemize}
Finally, set $(\cN_{k+1},\delta):= \bigcup_{\rho<\kappa}(\cN_{k,\rho},\delta)$.

\medskip

We note that $\bigcup_k\cN_k = \cN$, so we define $\delta$ on $\cN$ by $(\cN,\delta) = \bigcup_k(\cN_k,\delta)$. We claim that $\delta$ is generic. Let $A\subseteq N^{n+1}$ be an $L(N)$-definable set with $\dimL\big(\pi_n(A)\big) = n$. Then $A$ is $L(N_k)$-definable for some $k$, so there is $b \in N_{k+1}$ such that $\jet^n(b) \in A$.
\end{proof}

\begin{corollary}\label{Tderextensions}
$(\cM,\delta)$ can be extended to a model of $\TdG$.
\end{corollary}
\begin{proof}
Let $\cN$ be an elementary extension of $\cM$ with $\rkL(N|M) = |N| \geq |T|$ (such an extension exists, if $|N|>|M|$ then $\rkL(N|M) = |N|$). Now apply Proposition~\ref{extending}.
\end{proof}

\begin{corollary}\label{expansions}
Any $\cN \models T$ with $\rkL(N)\geq |T|$ admits an expansion to a model of $\TdG$. In particular, if $T$ is countable and has an Archimedean model then there is an expansion of $\R$ to a model of $\TdG$.
\end{corollary}
\begin{proof}
Apply Proposition~\ref{extending} with $(\bP,0)$ in place of $(\cM,\delta)$. If $T$ is countable and has an Archimedean model then by~\cite[Corollary 2.17]{DL95}, there is a unique expansion of $\R$ to a model of $T$. Since $T$ is countable, we have $\rkL(\R) = |\R|$.
\end{proof}

Corollary~\ref{expansions} generalizes a result of Brouette, who showed that $\R$ admits a derivation making it a closed ordered differential field~\cite{Br15}. 

\begin{remark}\label{converse}
We would conjecture that a partial converse to Corollary~\ref{expansions} holds as well: if a model $\cN\models T$ admits an expansion to a model of $\TdG$, then $\rkL(N)\geq \aleph_0$. This is true for $T = \RCF$, since by results of Rosenlicht~\cite{Ro74}, any sequence of distinct elements $(a_n)$ in a differential field of characteristic 0 with $a_n' = a_n^3-a_n^2\neq 0$ are necessarily algebraically independent. One can easily show that infinitely many such elements must exist in any model of $\TdG$. It remains to note that the $\dclL$-rank and the transcendence degree agree when $T = \RCF$.
\end{remark}

\subsection{The model completion of \texorpdfstring{$\Td$}{Tδ}}\label{subsec:TdG-MC}
In this subsection, we show that $\TdG$ is the model completion of $\Td$. We have already shown in Corollary~\ref{Tderextensions} that every model of $\Td$ extends to a model of $\TdG$. It remains to establish an embedding lemma.

\begin{lemma}\label{Tderembeddings}
Let $(\cM,\delta)\subseteq (\cN,\delta)\models \Td$, let $ (\cM,\delta)\subseteq (\cM^*,\delta)\models \TdG$ and suppose that $ (\cM^*,\delta)$ is $|N|^+$-saturated.
Then there is an $\Ld$-embedding $\iota:(\cN,\delta) \to (\cM^*,\delta)$ over $(\cM,\delta)$.
\end{lemma}
\begin{proof}
Take $a \in N\setminus M$. We note that $\cM\big\langle \jet^\infty(a)\big\rangle$ is closed under $\delta$ by Lemma~\ref{basicTderivation}, so it is a model of $\Td$. Without loss of generality, we assume that $\cN = \cM\big\langle \jet^\infty(a)\big\rangle$, as the general case follows by transfinite induction. 
We first consider the case that $a \in \cld(M)$. By Lemma~\ref{dcleq}, there is some minimal $n$ such that $\cN = \cM\big\langle \jet^{n-1}(a)\big\rangle$.
Let $f:M^n \to M$ be an $L(M)$-definable function such that $\delta^na = f\big( \jet^{n-1}(a)\big)$. We need to find $b \in M^*$ such that
\begin{enumerate}[(1)]
\item $\delta^nb = f\big( \jet^{n-1}(b)\big)$ and 
\item $\jet^{n-1}(b) \in B^{\cM^*}$ for every $L(M)$-definable set $B$ with $\jet^{n-1}(a) \in B^\cN$.
\end{enumerate}
If we can do this, then we can construct the embedding $\iota$ by sending $\jet^{n-1}(a)$ to $\jet^{n-1}(b)$.
By saturation, we may relax condition (2) and show that such a $b$ exists for an arbitrary $L(M)$-definable set $B$. Fix such a set $B$ and set $A:= \Gamma(f|_B)$. By minimality of $n$, the tuple $\jet^{n-1}(a)$ is $\dclL(M)$-independent, so $\dimL\big(\pi_n(A)\big) = \dimL(B) = n$. Since $(\cM^*,\delta) \models\TdG$, there is some $b \in M^*$ with $\jet^n(b)\in A$.

\medskip

Now consider the case that $a \not\in \cld(M)$. We need to find $b \in M^*$ such that
 $\jet^{n}(b) \in A^{\cM^*}$ for every $n$ and every $L(M)$-definable set $A$ with $\jet^{n}(a) \in A^\cN$. If we can do this, then we can construct the embedding $\iota$ by sending $\jet^{\infty}(a)$ to $\jet^{\infty}(b)$. Again by saturation, it suffices to do this for an arbitrary $n$ and an arbitrary $L(M)$-definable set $A$. If $\jet^{n}(a) \in A^\cN$, then $\dimL(A) = n+1$ since $\jet^n(a)$ is $\dclL(M)$-independent. Since $(\cM^*,\delta) \models\TdG$, there is some $b \in M^*$ with $\jet^n(b)\in A$.
\end{proof}

We can now prove our main theorem:

\begin{thm}\label{modelcompletion}
$\TdG$ is the model completion of $\Td$.
If $T$ has quantifier elimination and a universal axiomatization, then $\TdG$ has quantifier elimination.
\end{thm}
\begin{proof}
The fact that $\TdG$ is the model completion of $\Td$ follows from Corollary~\ref{Tderextensions}, Lemma~\ref{Tderembeddings} and Blum's criterion
(see the proof of Theorem 17.2 in~\cite{Sa10}; Lemma~\ref{Tderembeddings} shows that $\TdG$ has the stronger embedding property in that proof, allowing us to bypass the assumption that $\Td$ is universal).
If $T$ has quantifier elimination and a universal axiomatization, then each $L(\0)$-definable function is given piecewise by $L$-terms. The statement that $\delta$ is compatible with a given $L$-term $t$ is universal, so $\Td$ has a universal axiomatization. Thus, $\TdG$ has elimination of quantifiers, by~\cite[Theorem 13.2]{Sa10}.
\end{proof}

Theorem~\ref{modelcompletion} allows us to prove that $\TdG$ has an alternative axiomatization:

\begin{corollary}\label{altaxioms}
The following are equivalent
\begin{enumerate}[(1)]
\item $(\cM,\delta) \models \TdG$.
\item For each $n$ and each $L(M)$-definable set $X \subseteq M^{2n}$, if $\dimL\big(\pi_n(X) \big) = n$, then there is $\av \in M^n$ with $(\av,\delta \av) \in X$.
\end{enumerate}
\end{corollary}
\begin{proof}
Suppose that $(\cM,\delta) \models \TdG$ and fix an $L(M)$-definable set $X \subseteq M^{2n}$ with $\dimL\big(\pi_n(X) \big) = n$. Since $\pi_n(X)$ has nonempty interior, there is an extension $\cN\succeq_L \cM$ which contains a $\dclL(M)$-independent tuple $\bv \in \pi_n(X)^\cN$. By definable choice, there is an $L(M)$-definable map $f:\pi_n(X) \to M^n$ such that $\Gamma(f) \subseteq X$. 
By Lemma~\ref{transext}, there is a unique extension of $\delta$ to a $T$-derivation on $\cM\langle \bv\rangle$ such that $\delta \bv = f(\bv)$. Since $(\cM,\delta)$ is existentially closed in $\cM\langle \bv\rangle$ by Theorem~\ref{modelcompletion} there is $\av \in \pi_n(X)$ such that $\delta \av = f(\av)$. Then $(\av,\delta\av)\in X$.

\medskip

For the other direction, fix an $L(M)$-definable set $A \subseteq M^{n+1}$ with $\dimL\big(\pi_n(A) \big) = n$. Define $X\subseteq M^{2n}$ by
\[
X\ :=\ \big\{(\x,\y) \in M^{2n}: y_i = x_{i+1}\text{ for }i=1,\ldots,n-1\text{ and } (\x,y_n) \in A\big\}.
\]
Then $\pi_n(X) = \pi_n(A)$ and $(\av,\delta \av) \in X$ if and only if $\jet^n(a_1) \in A$ for any $\av = (a_1,\ldots,a_n) \in M^n$.
\end{proof}

\begin{corollary}
$\TdG$ is complete.
\end{corollary}
\begin{proof}
By extending $L$ by function symbols for all $L(\0)$-definable functions and by extending $T$ correspondingly, we may assume that $T$ has quantifier elimination and a universal axiomatization. Thus, $\TdG$ has quantifier elimination, so it suffices to show that it has a prime substructure. This follows from Corollary~\ref{primesubstructure}.
\end{proof}

Despite the fact that $\TdG$ has a prime substructure, it does not have a prime model. To prove this, we first show that each $\Ld$-formula is equivalent to a formula of a special form.

\begin{lemma}\label{formulaform}
For every $\Ld$-formula $\varphi$ there is some $m$ and some $L$-formula $\tilde{\varphi}$ such that
\[
\TdG\ \vdash\ \forall \x\big( \varphi(\x) \leftrightarrow \tilde{\varphi}\big(\jet^m(\x)\big)\big).
\]
\end{lemma}
\begin{proof}
Again by extending $L$, we may assume that $T$ has quantifier elimination and a universal axiomatization.
Then $\TdG$ has quantifier elimination, so we may assume that $\varphi$ is quantifier-free. Let $e(\varphi)$ be the number of times in $\varphi$ that $\delta$ is applied to a term that is not of the form $\delta^kx_i$. We proceed by induction on $e(\varphi)$. If $e(\varphi) = 0$ then we are done. If $e(\varphi) > 0$, then $\varphi$ is of the form
\[
\varphi(\x)\ =\ \psi\big(\x,\delta f(\jet^n(\x))\big)
\]
for some $n$, where $f(\y)$ is an $L(\0)$-definable function and where $\psi$ is an $\Ld$-formula. Let $\cD$ be an $L(\0)$-definable $\cC^1$-cell decomposition for $f$ (see Appendix~\ref{sec:Ck} for a precise definition). Then for each $D \in \cD$, there is an $L(\0)$-definable $\cC^1$-function $f_D$, defined in an open neighborhood of $D$, such that the $f_D(\y) = f(\y)$ for all $\y \in D$. Define $\tilde{f}$ by setting $\tilde{f}(\y) := \Jac_{f_D}(\y)$ whenever $\y$ in $D$. Then $\tilde{f}$ is $L(\0)$-definable and $\delta f(\y) = \tilde{f}(\y)\delta\y$ in any model of $\Td$. Set $\varphi'(\x):= \psi\big(\x,\tilde{f}\big(\jet^n(\x)\big)\delta\big(\jet^n(\x)\big)\big)$. Then $e(\varphi') < e(\varphi)$ and 
\[
\Td\ \vdash\ \forall\x\big(\varphi(\x) \leftrightarrow \varphi'(\x)\big). \qedhere
\]
\end{proof}

Note that $m$ and $\tilde{\varphi}$ in the lemma above are not unique.
The following corollary was established for $T= \RCF$ by Singer in~\cite{Si78}. Our proof is essentially the same.

\begin{corollary}
Suppose that $T$ is countable. Then $\TdG$ does not have a prime model.
\end{corollary}
\begin{proof}
Note that $\TdG$ is also countable. We use the fact that if the isolated types are not dense in the unary type space $S_1(\TdG)$, then $\TdG$ does not have a prime model (see~\cite[Proposition 32.1]{Sa10}). Given an $\Ld$-formula $\varphi(x)$, we let $[\varphi(x)]$ denote the clopen subset of $S_1(\TdG)$ consisting of all unary types containing $\varphi$. We claim that $[\delta x =1]$ contains no isolated types. Suppose towards contradiction that $[\psi(x)]$ is a basic clopen set contained in $[\delta x = 1]$ which isolates a type. By Lemma~\ref{formulaform}, we may assume that $\psi(x)$ is of the form $\tilde{\psi}\big(\jet^m(x)\big)$ for some $m$, where $\tilde{\psi}$ is a quantifier-free $L$-formula. Since $\psi(x)$ implies that $\delta x = 1$, we may replace $\delta x$ by 1 and $\delta^k x$ by 0 for all $k >1$. Thus, we may assume that $\psi(x)$ is actually an $L$-formula, so $\psi$ defines a finite union of points and open intervals. Since $[\psi]$ is assumed to isolate a type, the set defined by $\psi$ is just one point. However, this point lies in $\dclL(\0)$ in any model of $T$, so $\psi(x)$ implies $\delta(x) = 0$, a contradiction.
\end{proof}

Using Lemma~\ref{formulaform}, we have a nice description of $\Ld(M)$-definable sets:

\begin{corollary}\label{setform}
Let $(\cM,\delta) \models \TdG$ and let $B \subseteq M$ with $\delta B \subseteq B$. Then for any $\Ld(B)$-definable set $A \subseteq M^n$, there is some $m$ and some $L(B)$-definable set $\tilde{A}\subseteq M^{n(m+1)}$ such that
\[
A\ =\ \big\{\x \in M^n:\jet^m(\x) \in \tilde{A}\big\}.
\]
\end{corollary}
\begin{proof}
Take a tuple $\bv$ from $B$ and an $\Ld$-formula $\psi(\x,\y)$ such that $A = \psi(\cM,\bv)$. By Lemma~\ref{formulaform}, we have some $m$ and some $L$-formula $\tilde{\psi}$ such that
\[
(\cM,\delta)\ \models\ \forall \x\big( \varphi(\x,\bv) \leftrightarrow \tilde{\psi}\big(\jet^m(\x),\jet^m(\bv)\big).
\]
Set $\tilde{A} := \tilde{\psi}\big(\cM,\jet^m(\bv)\big)$.\qedhere
\end{proof}

\subsection{Distality and NIP}
Distal theories were introduced by Simon~\cite{Si13} as a subclass of NIP theories. The goal in this section is to show that $\TdG$ is distal.
Fix a monster model $(\M,\delta)\models \TdG$.

\begin{definition}
\label{distaldefformula}
$\TdG$ is \textbf{distal}
if whenever $\bv$ is a tuple from $\M$ and $(\av_i)_{i\in I}$ is an $\Ld(\0)$-indiscernible sequence from $\M$ such that
\begin{enumerate}[(1)]
\item $I = I_1+(c)+I_2$ where $I_1$ and $I_2$ are infinite without endpoints and
\item $(\av_i)_{i\in I_1+I_2}$ is $\Ld(\bv)$-indiscernible,
\end{enumerate}
then $(\av_i)_{i\in I}$ is $\Ld(\bv)$-indiscernible.
\end{definition}

\begin{thm}\label{distal}
$\TdG$ is distal.
\end{thm}
\begin{proof}
Fix an infinite linear order $I=I_1+(c)+I_2$ where $I_1$ and $I_2$ are infinite without endpoints and take an $\Ld(\0)$-indiscernible sequence $(\av_i)_{i \in I}$ from $\M^m$ and a tuple $\bv \in \M^n$ such that $(\av_i)_{i \in I_1+I_2}$ is $\Ld(\bv)$-indiscernible.
Let $\varphi(\x_1,\ldots,\x_n,\y)$ be an $\Ld$-formula. It suffices to show that
\[
\M\ \models\ \varphi(\av_{i_1},\ldots,\av_{i_n},\bv) \leftrightarrow \varphi(\av_{j_1},\ldots,\av_{j_n},\bv) 
\]
for any indices $i_1<\ldots<i_n$ and $j_1<\ldots<j_n$ in $I$.
 By Lemma~\ref{formulaform} there is some $m$ and some $L$-formula $\tilde{\varphi}$ such that
\[
\TdG\ \vdash\ \forall \x_1\ldots\forall \x_n\big( \varphi(\x_1,\ldots,\x_n,\y) \leftrightarrow \tilde{\varphi}\big(\jet^m(\x_1),\ldots,\jet^m(\x_n),\jet^m(\y)\big)\big).
\]
Since $(\av_i)_{i \in I}$ is $\Ld(\0)$-indiscernible, we have that $\big(\jet^m(\av_i)\big)_{i \in I}$ is also $\Ld(\0)$-indiscernible so, in particular, $\big(\jet^m(\av_i)\big)_{i \in I}$ is $L(\0)$-indiscernible. Likewise, since $(\av_i)_{i \in I_1+I_2}$ is $\Ld(\bv)$-indiscernible, we have that $\big(\jet^m(\av_i)\big)_{i \in I}$ is $L\big(\jet^m(\bv)\big)$-indiscernible. Since o-minimal theories are distal, we have that
\[
\M\ \models\ \tilde{\varphi}\big(\jet^m(\av_{i_1}),\ldots,\jet^m(\av_{i_n}),\jet^m(\bv)\big) \leftrightarrow \tilde{\varphi}\big(\jet^m(\av_{j_1}),\ldots,\jet^m(\av_{j_n}),\jet^m(\bv)\big).\qedhere
\]
\end{proof}

It is well-known that distality implies NIP (see~\cite[Remark 2.6]{CS16}).

\begin{corollary}\label{nip}
$\TdG$ has NIP.
\end{corollary}

The following negative result was first shown by Brouette~\cite{Br15}, who constructs a type of $\operatorname{dp}$-rank $\geq \aleph_0$. Our proof, which was suggested to us by Itay Kaplan, differs from the proof in~\cite{Br15}.

\begin{proposition}\label{notstrong}
$\TdG$ is not strongly dependent.
\end{proposition}
\begin{proof}
A NIP theory is strong if and only if it is strongly dependent, so we will show that $\TdG$ is not strong.
By~\cite{Ad07}, it is enough to find formulas $\varphi_k(x,y)$ and parameters $b_m$ from $\M$ such that
\begin{enumerate}[(i)]
\item $\varphi_k(\M,b_m)\cap \varphi_k(\M,b_n) = \0$ for all $k$ and all $m\neq n$ and
\item $\bigcap_k\varphi_k\big(\M,b_{f(k)}\big)\neq \0$ for any function $f:\N\to \N$.
\end{enumerate}
Let $\varphi_k$ be the formula 
\[
\varphi_k(x,y)\ :=\ y<\delta^k x <y+1
\]
and let $(b_m)$ be any increasing sequence such that $b_{m+1}-b_m>1$ for all $m$. Then clearly, $\varphi_k(x,b_m)$ and $\varphi_k(x,b_n)$ are incompatible for all $k$ and all $m\neq n$. Fix a function $f:\N\to \N$. By saturation, $\bigcap_k\varphi_k\big(\M,b_{f(k)}\big)$ is nonempty so long as any finite intersection $\bigcap_{k\leq n}\varphi_k\big(\M,b_{f(k)}\big)$ is nonempty. Set
\[
A\ :=\ \big(b_{f(0)},b_{f(0)}+1\big)\times \big(b_{f(1)},b_{f(1)}+1\big)\times \cdots\times \big(b_{f(n)},b_{f(n)}+1\big).
\]
Then $\pi_n(A)$ is open and so there is $b \in \M$ with $\jet^n(b) \in A$. Thus, $b \in \bigcap_{k\leq n}\varphi_k\big(\M,b_{f(k)}\big)$.
\end{proof}

\begin{remark}
When $T = \RCF$, Theorem~\ref{distal} was first noticed by Chernikov and has been employed to construct a distal extension of the theory of differentially closed fields of characteristic zero, see~\cite{ACGZ19}. Corollary~\ref{nip} was shown in~\cite{MR05}.
\end{remark}
\subsection{Dense pairs and closed ordered differential fields}\label{subsec:CODF}
Let $\cR \models \RCF$. In~\cite{Si78}, Singer axiomatizes the theory of closed ordered differential fields as follows:
\begin{definition}
$(\cR,\delta)$ is a \textbf{closed ordered differential field} if $\delta$ is a derivation on $\cR$ and if the following holds for each $n> 0$ and each each $P,Q_1,\ldots,Q_k \in R[X_1,\ldots,X_n]$: if $X_n$ does not appear in any of the $Q_i$ and if there is $\av \in R^{n}$ such that 
\[
P(\av)\ =\ 0,\ \ \frac{\partial P}{\partial X_{n}}(\av)\ \neq 0,\ \ \text{and each }Q_i(\av)\ >\ 0,
\]
 then there is $b \in R$ such that $P\big(\jet^n(b)\big) = 0$ and $Q_i\big(\jet^{n}(b)\big) >0$.
 Let $\CODF$ be the $\Ld$-theory axiomatizing closed ordered differential fields.
\end{definition}
Singer goes on to show that $\CODF$ is the model completion of ordered differential fields and, thus, of real closed ordered differential fields. 
Note that by Lemmas~\ref{isader} and~\ref{semialg}, $\delta$ is an $\RCF$-derivation if and only if it is a derivation, so $\RCF^\delta_G$ is the model completion of real closed ordered differential fields as well. By the uniqueness of model completions, we have the following:
\begin{proposition}
$(\cR,\delta)\models \RCF^\delta_G$ if and only if $(\cR,\delta) \models \CODF$.
\end{proposition}
The fact that Singer's axioms are equivalent to our more geometric axioms was first shown in~\cite{MR05}.

\medskip In~\cite{vdD98-2}, van den Dries introduced the theory of dense pairs of o-minimal structures:
\begin{definition}
$(\cN,P)$ is a \textbf{dense pair of models of $T$} if $\cN \models T$ and if $P(\cN)$ is the underlying set of a proper dense elementary substructure of $\cN$. Let $T^{\dense}$ be the $L\cup\{P\}$ theory axiomatizing dense pairs of models of $T$.
\end{definition}

If $(\cM,\delta) \models \TdG$, then the constant field $C$ is dense in $M$. To see this, let $I \subseteq M$ is an open interval. By the axioms of $\TdG$, there is some $c \in I$ with $\delta c = 0$. By Lemma~\ref{constantelem}, $C$ is the underlying set of an elementary $L$-substructure of $M$. Thus $(\cM,P)\models T^{\dense}$ where $P$ is interpreted to pick out $C$. Note that $(\cM,P)$ is a reduct of $(\cM,\delta)$ in the sense of definability. In~\cite{HN17}, Hieronymi and Nell show that dense pairs are not distal. However, since distality is not preserved under reducts, the question remains open as to whether models of $T^{\dense}$ have distal expansions. In light of Theorem~\ref{distal}, we are able to give a partial answer:

\begin{corollary}\label{distalexpansion}
There is a distal theory extending $T^{\dense}$.
\end{corollary}

It is worth noting that we do not have a method of expanding a given dense pair to a model of $\TdG$. Indeed, by Remark~\ref{converse} and the fact that there are models $(\cN,P)\models \RCF^P$ with $\rkL(N)<\aleph_0$, there are dense pairs which do not admit an expansion to a model of $\TdG$. Moreover, while dense pairs are defined for o-minimal theories extending the theory of divisible ordered abelian groups, Corollary~\ref{distalexpansion} is only a statement about o-minimal theories extending the theory of ordered fields.

\medskip

In the case that $T = \RCF$, Corollary~\ref{distalexpansion} was first observed by Cubides Kovacsics and Point~\cite{CP19}.
They study the expansions of dp-minimal fields by generic derivations and they show that distality is preserved in these expansions, using a method quite similar to ours.
All o-minimal theories are dp-minimal, but Cubides Kovacsics and Point do not require that their derivations are $T$-derivations and so the only common theory considered in this article and~\cite{CP19} is $\CODF$.

\medskip
In~\cite{vdD98-2}, van den Dries goes on to study the induced structure on $P(\cN)$ when $(\cN,P)\models T^{\dense}$. He shows that the only new sets introduced are the traces of $L(N)$-definable sets. 
We can show that if $(\cM,\delta) \models \TdG$ then the induced structure on the constant field is nothing more than this:
\begin{lemma}
Let $(\cM,\delta)$ be a model of $\TdG$ and let $C$ be its constant field.
For every $\Ld(M)$-definable set $A \subseteq C^n$, there is an $L(M)$-definable set $B\subseteq M^n$ such that $A = B\cap C^n$. 
\end{lemma}
\begin{proof}
By Corollary~\ref{setform}, there is some $m$ and some $L(M)$-definable set $\tilde{A}\subseteq M^{n(m+1)}$ such that
\[
A\ =\ \big\{\x \in M^n:\jet^m(\x) \in \tilde{A}\big\}.
\]
Since $A \subseteq C^n$, $\delta^k a = 0$ for any $a \in A$ and any $k >0$. Thus,
\[
A\ =\ C^n\cap \big\{\x \in M^n:(x_1,0,0,\ldots;x_2,0,0,\ldots;\ldots;x_n,0,0,\ldots) \in \tilde{A}\big\}.\qedhere
\]
\end{proof}

\section{Geometric and topological properties of \texorpdfstring{$\TdG$}{TδG}} \label{sec:geom}
In this section, we establish a dimension theory for models of $\TdG$ and a cell decomposition result. We show that $\TdG$ has $T$ as its open core and we use this to analyze the definable closure in models of $\TdG$. We also show that $\TdG$ eliminates imaginaries. For the remainder of this section, let $(\M,\delta)$ be a monster model of $\TdG$ and let $(\cM,\delta)$ be a small elementary substructure of $(\M,\delta)$.

\subsection{Dimension in models of $\TdG$} \label{subsec:dim}
In~\cite{Fo11}, the first author introduces the notion of an \textbf{existential matroid} and shows how these matroids induce a dimension function on definable sets. In this section, we apply these results.

\begin{lemma}\label{cldelem}
Let $B \subseteq \M$. If $\cld(B) = B$ then $(B,\delta|_B) \models \TdG$.
\end{lemma}
\begin{proof}
Since $\dclL(B) \subseteq \cld(B) = B$, we have that $B \preceq_L \M$. Since $\delta B \subseteq \cld(B) = B$ we also see that $B$ is closed under $\delta$, so $(B,\delta|_B)\models \Td$. Fix $n$ and some $L(B)$-definable set $A \subseteq \M^{n+1}$ with $\dimL\big(\pi_n(A) \big) = n$. We need to show that there is some $a \in B$ such that $\jet^n(a) \in A$. By definable choice, there is an $L(B)$-definable function $f: \pi_n(A) \to \M$ such that $\Gamma(f) \subseteq A$, so we may replace $A$ by $\Gamma(f)$. As $(\M,\delta) \models \TdG$, there is some $a \in \M$ such that $\jet^n(a) \in A$. Thus, $\delta^na \in \dclL\!\big(\jet^{n-1}(a)B\big)$ so $a \in \cld(B) = B$.
\end{proof}

The converse of Lemma~\ref{cldelem} does not hold as $\cld(M)\neq M$. To see this, let $C$ be the constant field of $\M$. Then $C \subseteq \cld(M)$, but by saturation, $|C|> |M|$, so $C$ is not contained in $M$. 

\begin{proposition}\label{existential}
$(\M,\cld)$ is an existential matroid.
\end{proposition}
\begin{proof}
$(\M,\cld)$ is a finitary matroid and, by~\cite[Lemma 3.23]{Fo11} and Lemma~\ref{cldelem}, $\cld$ satisfies \emph{existence}. It remains to show that $\cld$ is \emph{nontrivial} and that $\cld$ is \emph{definable}. To show that $\cld$ is nontrivial, we use saturation to find some $a \in \M$ such that $\jet^\infty(a)$ is $\dclL(\0)$-independent. Then $a \not\in \cld(\0)$, so $\cld(\0) \neq \M$. To show that $\cld$ is definable, it is enough to show for any $a \in \M$ and $B \subseteq \M$ that if $a \in \cld(B)$ then there is an $\Ld(B)$-definable set $A$ such that $a \in A \subseteq \cld(B)$. To see that this is true, we use (3) in~\ref{dcleq} to find some $L(\0)$-definable function $f$ such that $\delta^n a = f\big(\jet^n(a),\jet^m(\bv)\big)$ for some $n,m$ and some tuple $\bv$ from $B$. Then the $\Ld(\bv)$-definable set 
\[
A\ :=\ \big\{x \in \M:\delta^n x = f\big(\jet^n(x),\jet^m(\bv)\big)\big\}
\]
contains $a$ and is contained in $\cld(B)$.
\end{proof}

We now define a dimension function on the algebra of $\Ld(M)$-definable sets:
\begin{definition}\label{deltadim}
Let $A\subseteq \M^n$ be a nonempty $\Ld(M)$-definable set. We set
\[
\dimd(A)\ :=\ \max\big\{\rkd(\av|M):\av \in A\big\}
\]
and we call this the \textbf{$\delta$-dimension of $A$}. For completeness, we set $\dimd(\0) := -\infty$.
\end{definition}

By~\cite[Theorem 4.3]{Fo11}, this $\delta$-dimension satisfies the following axioms in~\cite{vdD89}:
\begin{enumerate}[(D1)]
\item $\dimd(\{a\}) = 0$ for each $a \in M$ and $\dimd(M) = 1$;
\item $\dimd(A\cup B) = \max\big\{\dimd(A),\dimd(B)\big\}$ for $\Ld(M)$-definable sets $A,B\subseteq M^n$;
\item $\dimd$ is invarient under permutation of coordinates.
\item If $A\subseteq M^{n+1}$ is $\Ld(M)$-definable, then the sets 
\[
A_i\ :=\ \big\{\x\in M^n:\dimd(A_{\x} )= i\big\}
\]
are $\Ld(M)$-definable for $i \in\{0,1\}$ and $\dimd\!\big(\{(\x,y)\in A: \x \in A_i\}\big) = \dimd(A_i)+i$.
\end{enumerate}

By~\cite[Proposition 1.7]{vdD89}, this dimension does not change if we pass to an elementary extension of $\M$, so this dimension does not depend on the choice of $\M$ and is invariant under elementary embeddings. We collect some consequences below, all of which are from~\cite{vdD89}:

\begin{corollary}
Let $A\subseteq M^m$ and $B\subseteq M^n$ be $\Ld(M)$-definable sets. The following hold:
\begin{enumerate}[(a)]
\item $\dimd(M^n) = n$;
\item $\dimd(A\times B) = \dimd(A) + \dimd(B)$;
\item If $m=n$ and $A\subseteq B$, then $\dimd(A)\leq \dimd(B)$;
\item If $A$ is finite and nonempty, then $\dimd(A) = 0$;
\item $f:A \to M^n$ is an $\Ld(M)$-definable map, then for each $i \in\{ 0,\ldots,m\}$, the set 
\[
B_i\ :=\ \big\{\x \in M^n: \dimd\!\big(f^{-1}(\x)\big)= i\big\}
\]
is $\Ld(M)$-definable and $\dimd\!\big(f^{-1}(B_i)\big) = \dimd(B_i)+i$. In particular, $\delta$-dimension is preserved under definable bijections.
\end{enumerate}
\end{corollary}

Finite sets are not the only sets of $\delta$-dimension 0. For example, the constant field of $(\cM,\delta)$ has $\delta$-dimension 0. To see this, we fix $a \in \M$ with $\delta a = 0$. Then for all $k$, the $\dclL$-rank $\rkL\big(\jet^k(a)|M\big) = 0$ if $a \in M$ and $\rkL\big(\jet^k(a)|M\big)= 1$ otherwise. In either case, we have by Corollary~\ref{practicalrk} that
\[
\rkd(a|M)\ =\ \lim\limits_{k \to \infty}\frac{\rkL\big(\jet^k(a)|M\big)}{k+1}\ =\ 0.
\]

\subsection{Cell decomposition}\label{subsec:cell-decomposition}
In~\cite{BMR09}, Brihaye, Michaux and Rivi\`{e}re prove a cell decomposition result for definable sets in closed ordered differential fields. As they remark in the final section of this paper, the only results that they use are quantifier elimination for $\CODF$, o-minimal cell decomposition for real closed ordered fields and the fact that the graph of $x\mapsto \jet^n(x)$ is dense in any model of $\CODF$. Thus, their results also apply to our case in light of the following lemma:

\begin{lemma}\label{density}
$\jet^m(M^n)$ is dense in $M^{n(m+1)}$ for any $(\cM,\delta) \models \TdG$.
\end{lemma}
\begin{proof}
Let $U_1,\ldots,U_n \subseteq M^{m+1}$ be basic (hence definable) open sets. Then by the axioms of $\TdG$, there is some $a_i \in M$ such that $\jet^{m}(a_i) \in U_i$ for each $i=1,\ldots,n$. Thus, $\jet^m(\av) \in U_1\times\cdots\times U_n$ where $\av = (a_1,\ldots,a_n)$.
\end{proof}

\medskip 

Brihaye, Michaux and Rivi\`{e}re say that a set $D\subseteq M^n$ is a \textbf{$\delta$-cell} if there is $m$ and an $L(M)$-definable cell $\tilde{D}\subseteq M^{n(m+1)}$ such that
\[
D\ =\ \big\{\x \in M^n:\jet^m(\x) \in \tilde{D}\big\}.
\]
Note that $D$ is $\Ld(M)$-definable. They call $\tilde{D}$ as above a \textbf{source cell} for $D$.
They assign to $D$ a binary sequence $(\bi_1;\ldots;\bi_n) \in \{0,1\}^n$, which they call the \textbf{$\delta$-type of $D$}, as follows: let $\tilde{D}$ be a source cell for $D$ and let 
\[
(i_{1,0},i_{1,1},\ldots,i_{1,m};i_{2,0},\ldots,i_{2,m};\ldots;i_{n,0},\ldots,i_{n,m})
\]
be the binary sequence associated to $\tilde{D}$. For $k = 1,\ldots,n$, set $\bi_k := 1$ if $i_{k,j} = 1$ for all $j = 0,\ldots,m$ and set $\bi_k:= 0$ otherwise. In Lemma 4.5 of~\cite{BMR09}, it is shown that this $\delta$-type is well-defined, i.e., it is independent of the choice of $m$ and $\tilde{D}$. A \textbf{$\delta$-cell decomposition of $M^n$} is a finite collection $\cD$ of disjoint $\delta$-cells such that $\bigcup\cD = M^n$ and such that $\{\pi_{n-1}D:D \in \cD\}$ is a $\delta$-cell decomposition of $M^{n-1}$.

\begin{thm}[Brihaye, Michaux, Rivi\`{e}re, Theorem 4.9]\label{decomp}
For any $B\subseteq M$ with $\delta B \subseteq B$ and any $\Ld(B)$-definable sets $A_1,\ldots,A_p \subseteq M^n$ there is an $\Ld(B)$-definable $\delta$-cell decomposition $\cD$ of $M^n$ partitioning $A_1,\ldots,A_p$.
\end{thm}

Brihaye, Michaux and Rivi\`ere use their cell decomposition theorem to define a dimension function (which they also call the $\delta$-dimension) on each $\Ld(M)$-definable subset $A$ of $(\cM,\delta)$. They go on to show that this dimension is equal to the maximum differential transcendence degree of a point contained in $A^\M$. Their argument can be adapted with virtually no change in proof to show that this dimension is equal to our $\delta$-dimension:

\begin{proposition}[Brihaye, Michaux, Rivi\`ere, Theorem 5.23]\label{dimcorres}
Let $A \subseteq M^n$ be an $\Ld(M)$-definable set and let $\dimd(A)$ be as in Definition~\ref{deltadim}. Then
\[
\dimd(A)\ =\ \max\big\{\bi_1+\cdots+\bi_n: A\text{ contains a }\delta\text{-cell of }\delta\text{-type }(\bi_1;\ldots;\bi_n)\big\}.
\]
\end{proposition}

As in the o-minimal case, this maximum is always realized in any $\delta$-cell decomposition partitioning $A$.
This correspondence gives us another way to compute the $\delta$-dimension of certain sets. For example, the constant field $C$ of $(\cM,\delta)$ is of the form
\[
C\ =\ \big\{x \in M:(x,\delta x) \in M \times \{0\}\}.
\]
Thus $C$ is a $\delta$-cell since $M \times \{0\}$ is a cell. The binary sequence associated to $M\times \{0\}$ is $(1,0)$, so the $\delta$-type of $C$ is $(0)$ and $\dimd(C) = 0$.

\subsection{Open core}\label{subsec:opencore}
Using
a theorem of Dolich, Miller and Steinhorn~\cite{DMS10}, Point shows that $\CODF$ has o-minimal open core~\cite{Po11}. While Point's proof works in our case as well, we can gather more information about the definable open sets by using a criterion developed by Boxall and Hieronymi~\cite{BH12}.
To use this criterion, we note that for each open $U\subseteq \M^n$ and each $\uv \in U$, the set
\[
\Big\{(\av,\bv)\in \M^{2n}: a_i<b_i\text{ for each $i$ and }\uv\in (a_1,b_1)\times \cdots\times (a_n,b_n)\subseteq U\Big\}
\]
clearly has nonempty interior, so ``Assumption (I)'' in~\cite{BH12} is satisfied in our setting.
For the remainder of this subsection, let $\av \in \M^n$ and let $B \subseteq \M$ be a small set with $\delta B \subseteq B$. Set
\[
\Xi_L(\av|B)\ :=\ \big\{\bv\in \M^n: \tp_L(\bv|B)= \tp_L(\av|B)\big\},\qquad\Xi_{\Ld}(\av|B)\ :=\ \big\{\bv\in \M^n: \tp_{\Ld}(\bv|B)= \tp_{\Ld}(\av|B)\big\}.
\]

\begin{lemma}\label{celldensity}
Suppose that $\rkd(\av|B) = n$ and let $X\subseteq \M^n$ be an $\Ld(B)$-definable set containing $\av$. Then there is an $L(B)$-definable open set $A \subseteq \M^n$ such that $\av \in A$ and $X \cap A$ is dense in $A$.
\end{lemma}
\begin{proof} By Corollary~\ref{setform}, there is some $m$ and some $L(B)$-definable set $\tilde{X} \subseteq \M^{n(m+1)}$ such that
\[
X\ =\ \big\{\x \in \M^n:\jet^m(\x) \in \tilde{X}\big\}.
\]
Let $\tilde{A} \subseteq \tilde{X}$ be an $L(B)$-definable cell containing $\jet^m(\av)$. Then $\tilde{A}$ must be open, since $\rkd(\av|B) = n$. Let $\pi:\M^{n(m+1)}\to \M^n$ be the projection map which maps
\[
(x_{1,0},x_{1,1},\ldots,x_{1,m};x_{2,0},\ldots,x_{2,m};\ldots;x_{n,0},\ldots,x_{n,m})\ \mapsto \ (x_{1,0};x_{2,0};\ldots;x_{n,0}).
\]
Then $\pi\big(\jet^m(\x)\big) = \x$ for all $\x \in \M^n$, so
\[
X\ =\ \pi\big(\jet^m(\M^n)\cap \tilde{X})\ \supseteq\ \pi\big(\jet^m(\M^n)\cap \tilde{A}).
\]
By Lemma~\ref{density}, we have that $\jet^m(\M^n)$ is dense in $\M^{n(m+1)}$, so $\jet^m(\M^n)\cap \tilde{A}$ is dense in $\tilde{A}$. This gives us that $X \cap \pi(\tilde{A})$ is dense in $\pi(\tilde{A})$, so we may set $A:= \pi(\tilde{A})$.
\end{proof}

\begin{lemma}\label{realizedrank}
$\rkd(\av|B) <n$ if and only if $\av$ is contained in some $\Ld(B)$-definable set of of $\delta$-dimension $<n$.
\end{lemma}
\begin{proof}
One direction follows immediately from our definition of $\delta$-dimension. For the other direction, suppose that $\rkd(\av|B) <n$. Then there is some $k\in \{1,\ldots,n\}$ and some $m$ such that 
\[
\delta^m a_k\ \in\ \dclL\!\big( \{\jet^m(a_1),\ldots,\jet^m(a_{k-1}),\jet^{m-1}(a_k)\}\cup B \big).
\]
Let $f: \M^{k(m+1)-1}\to \M$ be an $L(B)$-definable function such that
\[
f\big(\jet^m(a_1),\ldots,\jet^m(a_{k-1}),\jet^{m-1}(a_k)\big) \ =\ \delta^m(a_k).
\]
Then $\av$ is contained in the set 
\[
A\ :=\ \{\x \in \M^n:\jet^m(\x) \in \Gamma(f) \times \M^{(n-k)(m+1)}\big\}.
\]
It remains to note that $\dimd(A) \leq n-1$.
\end{proof}
\begin{lemma}\label{denseLtypes}
$\rkL(\av|B) = n$ if and only if $\Xi_L(\av|B)$ is open if and only if $\Xi_L(\av|B)$ is somewhere dense.
\end{lemma}
\begin{proof}
If $\rkL(\av|B)= n$, then any $L(B)$-definable set $X$ containing $\av$ contains an open neighborhood of $\av$. Let $(X_i)_{i \in I}$ be a list of all $L(B)$-definable sets containing $\av$, so $\Xi_L(\av|B) = \bigcap_{i \in I}X_i$. Fix $\bv \in \Xi_L(\av|B)$. Since $I$ is small and since $\bigcap_{i \in I_0}X_i$ contains an open neighborhood of $\bv$ for each finite $I_0 \subseteq I$, we can use saturation to find an open neighborhood $U$ of $\bv$ contained in $\bigcap_{i \in I}X_i$. Thus, $\bv$ is in the interior of $\Xi_L(\av|B)$. This shows that $\Xi_L(\av|B)$ is open and this of course implies that $\Xi_L(\av|B)$ is somewhere dense. 

\medskip

Now suppose that $\rkL(\av|B) < n$ and take some $L(B)$-definable set $X$ containing $\av$ with $\dimL(X)<n$. Then $X$ is nowhere dense and $\Xi_L(\av|B) \subseteq X$, so $\Xi_L(\av|B)$ is nowhere dense.
\end{proof}

\begin{lemma}\label{denseindeltatype}
If $\rkd(\av|B) = n$, then $\Xi_{\Ld}(\av|B)$ is dense in $\Xi_{L}(\av|B)$.
\end{lemma}
\begin{proof}
Fix $\bv \in \Xi_L(\av|B)$. We need to show that if $U \subseteq \M^n$ is an open set containing $\bv$, then $\Xi_{\Ld}(\av|B)\cap U$ is nonempty.
By saturation, it suffices to show that $U \cap X \neq \0$ for any $\Ld(B)$-definable set $X \subseteq \M^n$ containing $\av$. By Lemma~\ref{celldensity}, there is an $L(B)$-definable open set $A \subseteq \M^n$ such that $\av \in A$ and $X \cap A$ is dense in $A$. Since $\bv \in \Xi_L(\av|B)\subseteq A$, the intersection $U \cap A$ is nonempty and open, so $U \cap X$ is nonempty by density of $X\cap A$ in $A$.
\end{proof}

\begin{proposition}\label{opencore}
$\TdG$ has $T$ as its open core. More precisely, any open $\Ld(B)$-definable set is $L(B)$-definable.
\end{proposition} 
\begin{proof}
Let $A_n$ be the set of all $\av \in \M^n$ such that $\Xi_L(\av|B)$ is somewhere dense and let $A'_n$ be the set of all $\av \in A_n$ such that $\Xi_{\Ld}(\av|B)$ is dense in $\Xi_{L}(\av|B)$. By~\cite[Theorem 2.2]{BH12}, the proposition follows if we can show that $A'_n$ is dense in $\M^n$. Set
\[
D\ :=\ \big\{\av \in \M^n: \rkd(\av|B) = n\big\}.
\]
If $\av \in D$, then $\rkL(\av|B) = n$ and so $D \subseteq A_n$ by Lemma~\ref{denseLtypes}. By Lemma~\ref{denseindeltatype} we even have that $D \subseteq A_n'$, so it is enough to show that $D$ is dense in $\M^n$. By Lemma~\ref{realizedrank}, we have that 
\[
D\ =\ \M^n \setminus \bigcup\big\{X \subseteq \M^n: \text{$X$ is $\Ld(B)$-definable and $\dimd(X)<n$}\big\}
\]
Let $U \subseteq \M^n$ be a basic open set. By saturation, it suffices to show that $U \setminus X \neq \0$ for an arbitrary $\Ld(B)$-definable set $X$ of $\delta$-dimension $<n$. However, this follows easily from the fact that $\dimd(U)= n$.
\end{proof}

We list below two standard consequences of having o-minimal open core. See~\cite{DMS10} for proofs:

\begin{corollary}
$\TdG$ eliminates $\exists^\infty$ and every model of $\TdG$ is definably complete.
\end{corollary}

Moreover, we can use Proposition~\ref{opencore} to analyze the definable closure in models of $\TdG$:
\begin{corollary}\label{dclelem}
Let $A \subseteq M$. Then 
\[
\dcl_{\Ld}(A)\ =\ \dclL\big(\jet^\infty(A)\big).
\]
Thus $A$ is $\dcl_{\Ld}$-closed if and only if $(A,\delta|_A) \models \Td$.
\end{corollary}
\begin{proof}
Since $\jet^\infty(A)\subseteq\dcl_{\Ld}(A)$ and since $\dcl_{\Ld}(A)$ is $\dclL$-closed, we have $ \dclL\big(\jet^\infty(A)\big)\subseteq\dcl_{\Ld}(A)$. For the other direction, fix $a \in \dcl_{\Ld}(A)$. Since $\{a\}$ is closed and $\Ld(A)$-definable, we have that $\{a\}$ is $L\big(\jet^\infty(A)\big)$-definable by Proposition~\ref{opencore}.
\end{proof}

As is usual, this allows us to understand the $\Ld(B)$-definable functions:

\begin{corollary}
For any $\Ld(B)$-definable function $f:\M^n\to \M$ there is $m,k$ and $L(B)$-definable functions $\tilde{f}_1,\ldots,\tilde{f}_k:\M^{n(m+1)}\to \M$ such that for each $\x \in \M^n$, there is some $i \in \{1,\ldots,k\}$ such that
\[
f(\x)\ =\ \tilde{f}_i\big(\jet^m(\x)\big).
\]
\end{corollary}

\subsection{Elimination of imaginaries}
In this subsection, use the fact that $\TdG$ has $T$ as its open core and the fact that $T$ eliminates imaginaries to show that $\TdG$ eliminates imaginaries. This proof was communicated to us by Marcus Tressl. In~\cite{Po11}, Point used that $\CODF$ has o-minimal open core to prove that $\CODF$ eliminates imaginaries. Our method differs slightly, but her method also works in our case. 
Yet another proof of elimination of imaginaries for $\CODF$ can be found in~\cite{BCP19}.

\medskip

Fix an $\Ld(\M)$-definable set $A\subseteq \M^n$. We need to find a \textbf{canonical base} for $A$, i.e. a tuple $\av$ such that each $\Ld$-automorphism $\sigma: (\M,\delta) \to (\M,\delta)$ fixes $A$ setwise if and only if $\sigma$ fixes $\av$ componentwise. Let $\sigma$ be an arbitrary $\Ld$-automorphism of $(\M,\delta)$. By Corollary~\ref{setform}, there is some $m$ and some $L(M)$-definable set $B \subseteq M^{n(m+1)}$ such that 
\[
A\ = \ \big\{\x \in M^n: \jet^m(\x) \in B\big\}.
\]
By Proposition~\ref{opencore}, we have that $\overline{\jet^m(A)}$ is $L(\M)$-definable, where $\overline{\jet^m(A)}$ denotes the topological closure of $\jet^m(A)$. By replacing $B$ with $B\cap \overline{\jet^m(A)}$, we arrange that that $\jet^m(A)\subseteq B\subseteq \overline{\jet^m(A)}$. 
We associate to $A$ two other $\Ld(\M)$-definable sets:
\[
A^{\tcl}\ := \ \big\{\x \in M^n: \jet^m(\x) \in \overline{\jet^m(A)}\big\},\qquad A^{\fr}\ := \big\{\x \in M^n: \jet^m(\x) \in \overline{\jet^m(A)}\setminus B\big\}.
\]
Note that $A\cup A^{\fr} = A^{\tcl}$ and that $A\cap A^{\fr} = \0$. Note also that $\jet^m\big(\sigma(\bv)\big) = \sigma\big(\jet^m(\bv)\big)$ for all $\bv \in \M^n$.
\begin{lemma}\label{smaller}
$\dimL\big(\overline{\jet^m(A^{\fr})}\big)< \dimL( \overline{\jet^m(A)})$.
\end{lemma}
\begin{proof}
Set $B_0:= \overline{\jet^m(A)}\setminus B$, so $\jet^m(A^{\fr}) \subseteq B_0$. Since $\overline{B} = \overline{\jet^m(A)}$, we have that $\dimL(B_0)<\dimL( \overline{\jet^m(A)})$. Since the dimension of an $L(M)$-definable set doesn't increase when we take its closure, we get that 
\[
\dimL\big(\overline{\jet^m(A^{\fr})}\big)\ \leq\ \dimL(\overline{B_0})\ =\ \dimL(B_0)\ <\ \dimL( \overline{\jet^m(A)}).\qedhere
\]
\end{proof}

\begin{lemma}\label{twobases}
$\sigma(A) = A$ if and only if $\sigma(A^{\tcl}) = A^{\tcl}$ and $\sigma(A^{\fr}) = A^{\fr}$.
\end{lemma}
\begin{proof}
Suppose that $\sigma(A) = A$. Then $\sigma\big(\jet^m(A)\big) = \jet^m(A)$ and so $\sigma\big(\overline{\jet^m(A)}\big) =\overline{\jet^m(A)}$. We have
\[
\bv \in A^{\tcl}\ \Longleftrightarrow\ \jet^m(\bv) \in \overline{\jet^m(A)}\ \Longleftrightarrow\ \sigma\big(\jet^m(\bv)\big) \in \sigma\big(\overline{\jet^m(A)}\big) \ \Longleftrightarrow\ \sigma(\bv) \in A^{\tcl},
\]
since $\overline{\jet^m(A)}$ is $\sigma$-invariant.
Thus, $\sigma(A^{\tcl}) = A^{\tcl}$ and so $\sigma(A^{\fr}) = \sigma(A^{\tcl}\setminus A) = A^{\tcl}\setminus A = A^{\fr}$. For the other direction, we use that $\sigma(A) = \sigma(A^{\tcl}\setminus A^{\fr}) = \sigma(A^{\tcl})\setminus \sigma(A^{\fr})$.
\end{proof}

\begin{lemma}\label{closedbase}
If $A = A^{\tcl}$, then $A$ has a canonical base.
\end{lemma}
\begin{proof}
We first note that since $\overline{\jet^m(A)}$ is $L(M)$-definable and since $T$ eliminates imaginaries, there is a canonical base $\av$ for $\overline{\jet^m(A)}$. We claim that $\av$ is also a canonical base for $A$. We need to show that
\[
\sigma(A) = A\ \Longleftrightarrow \ \sigma\big(\overline{\jet^m(A)}\big) = \overline{\jet^m(A)}.
\]
First, if $\sigma(A) = A$ then $\sigma\big(\jet^m(A)\big) = \jet^m(A)$ and so $\sigma\big(\overline{\jet^m(A)}\big) = \overline{\jet^m(A)}$. Now, suppose that $\sigma\big(\overline{\jet^m(A)}\big) = \overline{\jet^m(A)}$ and fix $\bv \in A$. Then $\jet^m(\bv) \in \overline{\jet^m(A)}$ and so $\sigma\big(\jet^m(\bv)\big) \in \sigma\big(\overline{\jet^m(A)}\big) = \overline{\jet^m(A)}$, so $\sigma(\bv) \in A^{\tcl} = A$.
\end{proof}

\begin{thm}
$\TdG$ eliminates imaginaries.
\end{thm}
\begin{proof}
By Lemma~\ref{twobases}, it is enough to find canonical bases for $ A^{\tcl}$ and for $A^{\fr}$. By Lemma~\ref{closedbase}, there is a canonical base for $A^{\tcl}$. By Lemma~\ref{smaller}, we have $\dimL\big(\overline{\jet^m(A^{\fr})}\big)< \dimL( \overline{\jet^m(A)})$, so by induction on $\dimL( \overline{\jet^m(A)})$, we may assume that there is a canonical base for $A^{\fr}$ as well.
\end{proof}

\section{Several commuting \texorpdfstring{$T$}{T}-derivations} \label{sec:multi}
Let $\Delta = \{\delta_1, \ldots, \delta_p\}$ be a set of unary function symbols, let $\LD = L \cup \Delta$ and let $\TD$ be the $\LD$ theory extending $T$ by the following axiom schema:
\begin{enumerate}[(i)]
\item each $\delta_i$ is an $T$-derivation;
\item each $\delta_i$ and $\delta_j$ commute.
\end{enumerate}
The goal of this section is to show that $\TD$ has a model completion. When $T= \RCF$, this was shown by Rivi\`ere~\cite{Ri06}. Rivi\`ere's proof relies heavily on properties of differential polynomials, so we have to prove this another way. For the remainder of this section, we fix a model $(\cM,\Delta) \models \TD$.

\subsection{The monoid of derivative operators}
We use the notation in~\cite{Ko73} and denote by $\Theta$ the free abelian
monoid generated by $\Delta$. That is 
\[
\Theta\ =\ \big\{\delta_1^{e_1}\cdots\delta_p^{e_p}:e_1,\ldots,e_p \geq 0\big\}.
\]
We denote the identity element $\delta_1^{0}\cdots\delta_p^{0}$ by $\id$.
We view each $\theta \in \Theta$ as a function $a \mapsto \theta a:M\to M$ in the obvious way and, for a tuple $\av \in M^n$, we let $\theta\av := (\theta a_1,\ldots,\theta a_n)$.
To each $\theta = \delta_1^{e_1}\cdots\delta_p^{e_p}\in \Theta$, we set $\ord(\theta) := \sum_{i=1}^p e_i$ and we associate to $\theta$ the tuple $(\ord(\theta),e_1,\ldots,e_p) \in \N^{1+p}$. We put a (total) ordering $<$ on $\Theta$ by setting $\theta < \phi$ if the tuple corresponding to $\theta$ is less than the tuple corresponding to $\phi$ in the lexicographic order on $\N^{1+p}$. We note that $(\Theta,<)$ has order type $\omega$.

\medskip

We put another (partial) ordering $\prec$ on $\Theta$ by setting $\theta \prec \phi$ if there is $\xi \in \Theta$ with $\xi\theta = \phi$. Note that if $\theta \prec \phi$ then $\theta < \phi$, but the reverse does not hold. Both $(\Theta,<)$ and $(\Theta,\prec)$ are (partially) ordered monoids. We use $\preceq$ to denote the non-strict version of $\prec$. We note that $(\Theta,\prec)$ is in fact a lattice and for $\theta,\phi \in \Theta$, we let $\theta \lor \phi$ and $\theta \land \phi$ denote the $\prec$-supremum and $\prec$-infimum of $\theta$ and $\phi$, respectively. For any finite subset $P\subseteq \Theta$, we let $\bigvee P$ denote the $\prec$-supremum all $\theta \in P$, respectively.
 We let $\pred(\theta)$ denote the set of immediate $\prec$-predecessors of $\theta$. Then $\pred(\theta)$ is finite and nonempty so long as $\theta \neq \id$.

\medskip

For each $\theta \in \Theta$, we introduce new variables $y_1^\theta,\ldots,y_n^\theta$ and $z^\theta$. We use $y_j$ and $z$ in place of $y_j^{\id}$ and $z^{\id}$.
Given $J \subseteq \Theta$, we denote by $y_i^J$ the (possibly infinite) tuple of multivariables $(y_i^\theta)_{ \theta \in J}$ and we define $z^J$ analogously. We also set $\y:= (y_1,\ldots,y_n)$ and set $\y^J := (y_1^J,\ldots,y_n^J)$. Given $a \in M$, we set $a^J:= (\theta a)_{\theta \in J}$ and given a tuple $\bv \in M^n$, we set $\bv^J := (b_1^J,\ldots,b_n^J)$. We let $J^\ast = J \setminus \{\id\}$ and for $\phi \in \Theta$, we let 
\[
\phi J\ :=\ \{\phi\theta:\theta \in J\},\qquad J^{<\phi}\ :=\ \{\theta \in J:\theta < \phi\},\qquad J^{\prec\phi}\ :=\ \{\theta \in J:\theta \prec \phi\}.
\]
We view each subset $J$ of $\Theta$ as an ordered subset with ordering $<$. For example if $J$ is finite, then $z^J\cdot z^{\phi J} = \sum_{\theta \in J} z^\theta z^{\phi \theta}$. We often write a definable function $f$ as a function of infinitely many variables, i.e. $f= f(\y^\Theta)$. Of course, this just means that there is some finite set $J \subseteq \Theta$ such that $f$ only depends on the variables $\y^J$.

\begin{lemma}\label{commutecrit}
Let $(\cN,\Delta)\supseteq (\cM,\Delta)$ and let $A \subseteq N$ be a $\dclL(M)$-independent set with $\cN = \cM\langle A \rangle$. Then $(\cN,\Delta)\models \TD$ if and only if $\delta\epsilon a = \epsilon\delta a$ for all $\delta,\epsilon \in \Delta$ and for all $a \in A$.
\end{lemma}
\begin{proof}
One direction is trivial. For the other, fix $\delta,\epsilon \in \Delta$. Then $\gamma:=\delta\epsilon-\epsilon\delta$ is a $T$-derivation on $\cN$ by Lemma~\ref{twoderivs}, so we need to show that $\gamma$ is trivial. Any element in $N$ is of the form $f(\av)$ where $f$ is some $L(M)$-definable function and where $\av$ is a tuple from $A$. Since $\av$ is $\dclL(M)$-independent, there is some open set $U$ containing $\av$ such that $f$ is $\cC^1$ on $U$. We have that 
\[
\gamma f(\av)\ =\ f^{[\gamma]}(\av)+\Jac_f(\av)\gamma \av.
\]
By the assumption that $\delta$ and $\epsilon$ commute on $A$, we have that $\gamma\av = 0$ and since $M \subseteq \ker(\gamma)$, we also have that $f^{[\gamma]} = 0$ by Lemma~\ref{basicTderivation}. Therefore, $\gamma f(\av) = 0$.
\end{proof}

\begin{lemma}\label{cordelta}
Let $k\geq 1$. Given $\delta \in \Delta$ and an $L(M)$-definable $\cC^k$-function $f:U \to M$ with $U \subseteq M^n$ open, there is an $L(M)$-definable $\cC^{k-1}$-function function $f^\delta: U \times M^n \to M$ such that:
\begin{enumerate}[(1)]
\item If $(\cM,\delta) \supseteq (\cN,\delta)\models \Td$ and $\uv\in U^\cN$, then 
\[
\delta f(\uv)\ =\ f^\delta(\uv,\delta\uv).
\]
\item If $k \geq 2$, then 
\[
(f^\delta)^\epsilon(\y,\y^\delta,\y^\epsilon,\y^{\epsilon\delta})\ =\ (f^\epsilon)^\delta(\y,\y^\epsilon,\y^\delta,\y^{\epsilon\delta})
\]
for all $\epsilon \in \Delta$.
\item If $g :V\to U$ is an $L(M)$-definable $\cC^k$-map with $V \subseteq M^m$ open, then for $h := f\circ g$, we have
\[
h^{\delta}(\y,\y^\delta)\ =\ f^{\delta}\big(g(\y),g^\delta(\y,\y^\delta)\big)
\]
where $g^{\delta} = (g_1^{\delta},\ldots,g_n^{\delta})$.
\end{enumerate}
\end{lemma}
\begin{proof}
We define $f^\delta$ by 
\[
f^{\delta}(\y,\y^\delta)\ :=\ f^{[\delta]}(\y)+\Jac_f(\y) \y^\delta.
\]
Then (1) follows immediately from Lemma~\ref{basicTderivation}. For (2), fix $\epsilon \in \Delta$ and suppose that $f$ is $\cC^k$ for some $k \geq 2$. By the proof of Lemma~\ref{basicTderivation}, there is an $L(\0)$-definable $\cC^k$-function $F(\x,\y)$ with open domain and a tuple $\av$ such that $F(\av,\uv) = f(\uv)$ for all $\uv \in U$. By the proof of Lemma~\ref{twoderivs}, we have that
\[
(f^\delta)^\epsilon(\y,\y^\delta,\y^\epsilon,\y^{\epsilon\delta}) -(f^\epsilon)^\delta(\y,\y^\epsilon,\y^\delta,\y^{\delta\epsilon})\ =\ \Jac_F(\av,\y)(\delta\epsilon\av-\epsilon\delta \av, \y^{\delta\epsilon}-\y^{\delta\epsilon})\ =\ 0.
\]
As for (3), let $F$ and $\av$ be as above and take an $L(\0)$-definable $\cC^k$-map $G$ with open domain and a tuple $\bv$ such that $G(\bv,\uv) = g(\uv)$ for all $\uv \in V$. By shrinking the domain of $G$, we may assume that the range of $G$ is contained in the domain of $F$. 
Then $F\big(\av,G(\bv,\uv)\big) = h(\uv)$ for all $\uv \in V$. We have
\begin{align*}
h^{\delta}(\y,\y^\delta)\ &=\ \Jac_F\big(\av,G(\bv,\y)\big) \big(\delta \av,\Jac_{G}(\bv,\y) (\delta\bv,\y^\delta)\big)\\
&=\ \frac{\partial F}{\partial \x}\big(\av,g(\y)\big)\delta\av+\frac{\partial F}{\partial \y}\big(\av,g(\y)\big)\left(\frac{\partial G}{\partial \x}(\bv,\y)\delta\bv+\frac{\partial G}{\partial \y}\y^\delta\right)\\ 
&=\ f^{[\delta]}\big(g(\y)\big)+ \Jac_f\big(g(\y)\big) \big(g^{[\delta]}(\y)+ \Jac_g(\y)\y^\delta\big)\ =\ f^{\delta}\big(g(\y),g^\delta(\y,\y^\delta)\big).\qedhere
\end{align*}
\end{proof}

\subsection{Coherent conditions}\label{subsec:coherent}
Let $P \subseteq \Theta^*$ be a (possibly empty) set of pairwise $\prec$-incomparable elements. We set 
\[
B\ :=\ \big\{\theta \in \Theta: \beta \preceq \theta\text{ for some }\beta \in P\big\}.
\]
Then $P$ is precisely the set of $\prec$-minimal elements of $B$, hence finite by Dickson's Lemma.
We set $I:= \Theta \setminus B$.
 A \textbf{condition (on $\cM$)} is a tuple $\cC = \big(P,U, (f_\beta)_{\beta \in P}\big)$ where $P$ is as above such that:
\begin{enumerate}[(i)]
\item $U \subseteq M^I$ is a nonempty, open, $L(M)$-definable set and
\item each $f_\beta:U \to M$ is an $L(M)$-definable continuous function which only depends on the variable $z^\theta$ if $\theta< \beta$.
\end{enumerate}

Given $a \in M$, we say that \textbf{$a$ satisfies $\cC$} if $a^I \in U$ and if $\beta a = f_\beta(a^I)$ for all $\beta \in P$. We note that all but finitely coordinate projections of $U$ are not all of $M$, so this is a finitary statement even though $I$ may be infinite.
We think of a condition as describing the algebraic dependencies among components of the tuple $a^\Theta$: the tuple $a^I$ is seen as being \emph{independent} and $a^B$ is seen as being \emph{bounded}. The dependencies of the components of $a^B$ are uniquely determined by requiring that $\beta a= f_\beta(a^I)$ whenever $\beta \in P$. Of course, most conditions simply can not be satisfied in a model of $\TD$, so we must put some extra compatibility requirements on our conditions.

\medskip 

Fix a condition $\cC$. We will assign to each $\theta \in \Theta$ an $L(M)$-definable open set $U_\theta \subseteq M^I$, a set $\Omega_\theta$ of $L(M)$-definable continuous functions on $U_\theta$ and a distinguished function $g_\theta \in \Omega_\theta$. We require that the following properties are satisfied:
\begin{itemize}
\item $U_\theta$ is a dense open subset of $U$ and $U_\theta \subseteq U_{\phi}$ whenever $\theta < \phi$,
\item each $h \in \Omega_\theta$ only depends on the variable $z^\phi$ if $\phi \leq \theta$.
\item each $\Omega_\theta$ is finite.
\end{itemize}
We define $\Omega_\theta$, $U_\theta$ and $g_\theta$ inductively. For $J \subseteq \Theta$, we let $g_J = (g_\phi)_{\phi \in J}$, assuming that each $g_\phi$ has been defined.
Set $U_{\id} := U$ and set $\Omega_{\id} := \{z\}$, so $g_{\id} = z$. Suppose that $U_\phi$, $\Omega_\phi$ and $g_\phi$ are defined for each $\phi < \theta$ and let $\theta'$ be the immediate $<$-predecessor of $\theta$.
\begin{enumerate}[(1)]
\item If $\theta \in I$, set $U_\theta := U_{\theta'}$ and set $\Omega_\theta := \{z^\theta\}$, so $g_\theta = z^\theta$.
\item If $\theta \in P$, then we have a distinguished function $f_\theta$ given by the condition $\cC$. Set $U_\theta := U_{\theta'}$ and set $\Omega_\theta := \{ f_\theta|_{U_{\theta}}\}$, so $g_\theta := f_\theta|_{U_{\theta}}$.
\item If $\theta \in B \setminus P$, then $\pred(\theta) \cap B$ is nonempty and finite. Set
\[
U_\theta\ :=\ \big\{\uv \in U_{\theta'}: g_\phi\text{ is }\cC^1 \text{ at }\uv\text{ for all }\phi \in \pred(\theta) \cap B\big\}
\]
Then $U_\theta$ is a dense open subset of $U_{\theta'}$. Now fix $\phi \in \pred(\theta) \cap B$, so $\theta = \delta\phi $ for some $\delta \in \Delta$. 
Set $J:=I^{\leq \phi}= I^{<\phi}$, so $g_\phi$ only depends on $z^J$. Then $\delta J< \delta\phi = \theta$, so $g_{\delta J}$ has already been defined. We define $g_{\theta,\phi}:U_\theta\to M$ by 
\[
g_{\theta,\phi}(z^I)\ :=\ g_\phi^{\delta}\big(z^I,g_{\delta J}(z^I)\big).
\]
We set $\Omega_\theta:= \{g_{\theta,\phi}:\phi \in \pred(\theta) \cap B\}$ and we let $g_\theta$ be an arbitrary element of $\Omega_\theta$.
\end{enumerate}

\begin{definition}
We say that $\cC$ is \textbf{coherent} if $\Omega_\theta$ is a singleton for all $\theta\leq \bigvee P$.
\end{definition}

\begin{proposition}\label{strcoh}
If $\cC$ is coherent, then $\Omega_\theta$ is a singleton for all $\theta$.
\end{proposition}
\begin{proof}
Suppose towards contradiction that there is $\theta \in \Theta$ such that $\Omega_\theta$ is not a singleton. Let $\theta$ be $<$-minimal with this property. Then $\theta$ is in $B \setminus P$ and there are distinct $\phi_1,\phi_2 \in \pred(\theta) \cap B$ such that $g_{\theta,\phi_1} \neq g_{\theta,\phi_2}$.
We first claim that there is $\phi_0 \in \pred(\theta)\cap B$ such that $\phi_0\land \phi_i \in B$ for $i = 1,2$.
Since $\phi_1$ and $\phi_2$ are elements of $B$, there are $\beta_1,\beta_2 \in P$ such that $\beta_i \preceq \phi_i$ for $i = 1,2$. If $\beta_1 = \beta_2$, then $\phi_1 \land \phi_2 \succeq \beta_1$ so we are done (let $\phi_0:= \phi_1)$. Thus, we assume that $\beta_1$ and $\beta_2$ are distinct. Since $\beta_1,\beta_2 \prec \theta$ and since $\theta > \bigvee P$, we have that $\beta_1\lor\beta_2 \prec \theta$. Therefore, there is $\phi_0 \in \pred(\theta)$ with $\beta_1\lor\beta_2 \preceq \phi_0$. It remains to observe that $\phi_0 \land \phi_i \succeq \beta_i$, so $\phi_0 \land \phi_i \in B$ for $i = 1,2$.

\medskip

We will now show that $g_{\theta,\phi_0} = g_{\theta,\phi_1}$. Fix $\delta,\epsilon \in \Delta$ such that $\theta = \delta\phi_0 = \epsilon\phi_1$ and set $\gamma := \phi_0\land \phi_1$. Then $\phi_0 = \epsilon\gamma$ and $\phi_1 = \delta\gamma$. Set $J := I^{<\gamma}$, so $\epsilon\delta J<\theta$ and, by minimality of $\theta$, we have that
 $\Omega_\alpha$ is a singleton whenever $\alpha$ is in in $\delta J$, $\epsilon J$ or $\epsilon\delta J$.
We set 
\[
V\ :=\ \big\{\uv \in U_\theta: g_\gamma \text{ is } \cC^2\text{ at }\uv\text{ and }g_{\delta J},g_{\epsilon J}\text{ are }\cC^1\text{ at }\uv\big\}.
\]
Then $V$ is an open dense subset of $U_\theta$ and since both $g_{\theta,\phi_0}$ and $g_{\theta,\phi_1}$ are continuous, it suffices to show that they are equal on $V$. We work in $V$ for the remainder of the proof. We have
\[
g_{\phi_0}\ =\ g_{\phi_0,\gamma}\ =\ g_\gamma^\epsilon\big(z^I,g_{\epsilon J}(z^I)\big)
\]
and so, by Lemma~\ref{cordelta} (3), we have $g_{\phi_0}^\delta(z^I,z^{\delta I}) = (g_\gamma^\epsilon)^\delta\big(z^I,g_{\epsilon J}(z^I),z^{\delta I},g_{\epsilon J}^\delta(z^I,z^{\delta I})\big)$.
Thus, 
\[
g_{\theta,\phi_0}(z^I)\ =\ (g_\gamma^\epsilon)^\delta\big(z^I,g_{\epsilon J}(z^I),g_{\delta J}(z^I),g^\delta_{\epsilon J}(z^I,g_{\delta J}(z^I))\big)\ =\ (g_\gamma^\epsilon)^\delta\big(z^I,g_{\epsilon J}(z^I),g_{\delta J}(z^I),g_{\epsilon\delta J}(z^I)\big).
\]
Likewise, we have 
\[
g_{\theta,\phi_1}(z^I)\ =\ (g_\gamma^\delta)^\epsilon\big(z^I,g_{\delta J}(z^I),g_{\epsilon J}(z^I),g_{\epsilon\delta J}(z^I)\big),
\]
and so $g_{\theta,\phi_0} = g_{\theta,\phi_1}$ by Lemma~\ref{cordelta} (2). The same argument shows that $g_{\theta,\phi_0} = g_{\theta,\phi_2}$, a contradiction.
\end{proof}

\begin{lemma}\label{trivfact}
If $\cC$ is coherent, then $g_{\delta\theta}(z^I) = g^\delta_{\theta}(z^I,g_{\delta I}(z^I))$ for all $\theta \in \Theta$ and all $\delta \in \Delta$.
\end{lemma}
\begin{proof}
This follows from Proposition~\ref{strcoh} if $\theta \in B$. If $\theta \in I$, then $g_\theta(z^I) = z^\theta$ so $g_\theta^\delta(z^I,z^{\delta I}) = z^{\delta\theta}$. Thus $g^\delta_{\theta}(z^I,g_{\delta I}(z^I)) = g_{\delta\theta}(z^I)$.
\end{proof}

\subsection{The model completion of $\TD$}
We say that $\Delta$ is a \textbf{set of generic commuting derivations} if every coherent condition on $\cM$ is satisfied by some $a \in M$. Let $\TDG$ be the $\LD$-theory extending $\TD$ by the axiom scheme which asserts that $\Delta$ is a generic set of commuting derivations.
This subsection is dedicated to showing that $\TDG$ is the model completion of $\TD$. We need two lemmas.

\begin{lemma}
Any model of $\TD$ can be extended to a model of $\TDG$.
\end{lemma}
\begin{proof}
Let $\cC = \big(P,U,(f_\beta)_{\beta\in P}\big)$ be a coherent condition on $\cM$ and let $I,B \subseteq \Theta$ and $(U_\theta)_{\theta \in \Theta}$, $(g_\theta)_{\theta \in \Theta}$ be as in the previous subsection. We will construct a model $(\cN,\Delta)\models \TD$ extending $(\cM,\Delta)$ such that there is $a \in N$ satisfying $\cC$. First, let $\cN \succeq_L \cM$ be an elementary extension which contains a $\dclL(M)$-independent tuple $a_I :=(a_\theta)_{\theta \in I}$ with $a_I \in U^\cN$. We may assume that $\cN = \cM \langle a_I\rangle$. Using Lemma~\ref{transext}, we extend each $\delta \in \Delta$ to a $T$-derivation on $\cN$ such that 
\[
\delta a_\theta\ :=\ g_{\delta\theta}(a_I)
\]
for all $\theta \in I$.
Since $a := a_{\id}$ satisfies $\cC$, it remains to show that our extended $T$-derivations commute. 
Let $\delta,\epsilon \in \Delta$ and $\theta \in I$ be arbitrary. By Lemma~\ref{commutecrit}, it suffices to show that $\delta\epsilon a_\theta = \epsilon\delta a_\theta$. We have $\epsilon a_\theta = g_{\epsilon\theta}(a_I)$ and so $\delta\epsilon a_\theta = g_{\epsilon\theta}^\delta(a_I,\delta a_I)$ Since $\delta a_\phi = g_{\delta\phi}(a_I)$ for each $\phi \in I$, we have
\[
\delta\epsilon a_\theta\ =\ g_{\epsilon\theta}^\delta\big(a_I, g_{\delta I}(a_I)\big)\ =\ g_{\delta\epsilon\theta}(a_I)
\]
by Lemma~\ref{trivfact}. Likewise, $\epsilon\delta a_\theta = g_{\delta\epsilon\theta}(a_I)$.
\end{proof}

\begin{lemma}
Let $(\cM,\Delta)\subseteq (\cN,\Delta)\models \TD$, let $ (\cM,\Delta)\subseteq (\cM^*,\Delta)\models \TDG$ and suppose that $ (\cM^*,\Delta)$ is $|N|^+$-saturated. Then there is an $\LD$-embedding $\iota:(\cN,\Delta) \to (\cM^*,\Delta)$ over $(\cM,\Delta)$.
\end{lemma}
\begin{proof}
We may assume that $\cN = \cM\langle a^\Theta\rangle$ for some $a \in N \setminus M$. Let $\theta_0,\theta_1,\ldots,\theta_n,\ldots$ be the enumeration of $\Theta$ with respect to $<$. We build an increasing chain of subsets $I_n\subseteq \Theta$ as follows:
\begin{itemize}
\item Set $I_0 = \{\theta_0\} = \{\id\}$.
\item If $I_n$ has already been defined and if $\theta_{n+1}a \not\in \dclL(M a^{I_n})$, then set $I_{n+1}:= I_n\cup \{\theta_n\}$. If $\theta_{n+1}a \in \dclL(a^{I_n})$, then set $I_{n+1}:= I_n$.
\end{itemize}
Set $I := \bigcup_n I_n$. By construction, we have that $a^I$ is a maximal $\dclL(M)$-independent subtuple of $a^\Theta$. If $\theta \not\in I$ then $\delta\theta \not\in I$ for all $\delta \in \Delta$, so $I$ is $\prec$-downward closed. Set $B := \Theta\setminus I$ and let $P$ be the (finite) set of $\prec$-minimal elements of $B$. 
If $\theta_n \in P$ for some $n$, then $\theta_n \not\in I_n$, so we have $\theta_n a \in \dclL(Ma^{I_n})$. We let $f_{\theta_n}:M^{I_n} \to M$ be an $L(M)$-definable function such that $\theta_n a = f_{\theta_n}(a^{I_n})$ and we view $f_{\theta_n}$ as a function on all of $M^I$. Note that the quantifier-free $\LD$-type of $a$ over $(\cM,\Delta)$ is completely characterized by the $L(M)$-definable sets which contain $a^I$ and by the fact that $\beta a = f_\beta(a^I)$ for $\beta \in P$. 

\medskip

Let $U\subseteq M^I$ be an $L(M)$-definable set with $a^I\in U^{\cN}$. Then $U$ has nonempty interior, so by shrinking $U$ we may assume that $U$ is open and that $f_\beta$ is continuous on $U$ for all $\beta \in P$. Thus, $\big(P,U,(f_\beta)_{\beta \in P}\big)$ is a condition on $\cM$ which is satisfied by $a$, but this condition may not be coherent. We resolve this issue as follows: let $(\Omega_\theta)_{\theta \in \Theta}$ be as in the previous subsection. A quick inductive argument shows that $\theta a = h(a^I)$ for any $\theta \in \Theta$ and any $h \in \Omega_\theta$. Thus, all of the functions on $\Omega_\theta$ agree at $a^I$ and, by the $\dcl_L(M)$-independence of $a^I$, they all agree on some $L(M)$-definable open set $U_\theta \subseteq U$. Set 
\[
V\ :=\ \bigcap_{\theta < \bigvee P} U_\theta.
\]
Then $\big(P,V,(f_\beta)_{\beta \in P}\big)$ is a coherent condition on $\cM$ which is satisfied by $a$ and, as $(\cM^*,\Delta) \models \TDG$, it is also satisfied by some element of $M^*$. Since $U$ was arbitrary, we may use the saturation of $(\cM^*,\Delta)$ to find some $b \in M^*$ such that $b^I$ is contained in exactly the same $L(M)$-definable sets as $a^I$ (in their respective models) and such that $\beta b = f_\beta(b^I)$ for $\beta \in P$.
\end{proof} 

Arguing as in the proof of Theorem~\ref{modelcompletion}, we have the following:
\begin{thm}\label{multmodelcompletion}
$\TDG$ is the model completion of $\TD$.
If $T$ has quantifier elimination and a universal axiomatization, then $\TDG$ has quantifier elimination.
\end{thm}

We can immediately reprove some of our previous results in this more general setting. The proof of the following is virtually the same as the proof of Lemma~\ref{formulaform}:

\begin{lemma}\label{formulaformD}
For every $\LD$-formula $\varphi$ there is some finite $J \subseteq \Theta$ and some $L$-formula $\tilde{\varphi}$ such that
\[
\TDG\ \vdash\ \forall \y\big( \varphi(\y) \leftrightarrow \tilde{\varphi}\big(\y^J \big)\big).
\]
\end{lemma}

We may substitute the above lemma in place of Lemma~\ref{formulaform} in the proof of Theorem~\ref{distal} to show:

\begin{proposition}\label{distalD}
$\TDG$ is distal.
\end{proposition}

One potentially interesting structure is the reduct $(\cM,P_1,\ldots,P_p)$ of a model $(\cM,\Delta)\models \TDG$, where $P_i$ is interpreted to pick out the constant field $\ker(\delta_i)$. We have $(\cM,P_i) \models T^{\dense}$ for each $i$, but to our knowledge, no work has been done on $(\cM,P_1,\ldots,P_p)$.
\subsection{Dimension in models of $\TDG$}
In this subsection, we define and examine the $\Delta$-closure in analogy with the $\delta$-closure in \S\ref{sec:dims}. Given $B \subseteq M$ and $J \subseteq \Theta$, we set $B^J:= \{b^J:b \in B\}$.

\begin{definition}
Given $a \in M$ and $B\subseteq M$, we say that \textbf{$a$ is in the $\Delta$-closure of $B$}, written $a \in \clD(B)$, if $a^\Theta$ is \emph{not} $\dclL(B^\Theta)$-independent.
\end{definition}

The next fact follows from the finitary nature of $\dclL$:
\begin{fact}\label{clDalt}
$a \in \cl^\Delta(B)$ if and only if there is some finite $J \subseteq \Theta$ such that
\[
\rkL(a^J|B^{\Theta})\ <\ |J|.
\]
\end{fact}

We can examine the $\Delta$-closure by induction on $|\Delta|$. The following lemma serves as an induction step:

\begin{lemma}\label{inductionstep}
Let $\delta\in \Delta$ and set $\Delta_0:= \Delta \setminus \{\delta\}$. Then $\delta$ is a quasi-endomorphism of $(\cM,\clDz)$.
\end{lemma}
\begin{proof}
Fix $A,B \subseteq M$. Making the same reduction as in Lemma~\ref{isquasi}, it suffices to show that if $A \subseteq \clDz(B)$ then $\delta A \subseteq \clDz(B\delta B)$. Fix $a \in A$ and let $\Theta_0 \subseteq \Theta$ be the submonoid of $\Theta$ generated by $\Delta_0$. Since $a \in \clDz(B)$, there is some finite $J \subseteq \Theta_0$ such that
\[
\rkL(a^J|B^{\Theta_0})\ <\ |J|.
\]
Since $\delta$ is a quasi-endomorphism of $(\cM,\dclL)$ by Lemma~\ref{isquasi}, we have that
\[
\rkL\big(\delta(a^{ J})|a^{J}B^{\Theta_0}\delta(B^{ \Theta_0})\big)\ <\ |J|.
\]
Since $\delta$ commutes with all $\theta \in \Theta_0$, we have $\delta(a^{ J}) = (\delta a)^{J}$. Likewise, $\delta(B^{ \Theta_0}) = (\delta B)^{\Theta_0}$, so 
\[
\rkL\big((\delta a)^{ J}|a^J(B\delta B)^{\Theta_0}\big)\ =\ \rkL\big((\delta a)^{ J}|(B\delta B)^{\Theta_0}\big)\ <\ |J|.
\]
Thus, $\delta a \in\clDz(B\delta B)$.
\end{proof}

\begin{proposition}
$(\cM,\clD)$ is a finitary matroid.
\end{proposition}
\begin{proof}
By induction on $|\Delta|$. Fix $\delta\in \Delta$ and set $\Delta_0:= \Delta \setminus \{\delta\}$. Then by our induction hypothesis, $(\cM,\clDz)$ is a finitary matroid and by Lemma~\ref{inductionstep}, $\delta$ is a quasi-endomorphism of $(\cM,\clDz)$. Fix $a \in M$ and $B\subseteq M$. We claim that $a \not\in \clD(B)$ if and only if $\jet^\infty(a)$ is $\clDz\!\big(\jet^\infty(B)\big)$-independent. To see this, let $\Theta_0$ be as in the proof of Lemma~\ref{inductionstep}. Then $\jet^\infty(a)$ is $\clDz\!\big(\jet^\infty(B)\big)$-independent if and only if $\jet^\infty(a)^{\Theta_0}$ is $\dclL\big(\jet^\infty(B)^{\Theta_0}\big)$-independent if and only if $a^\Theta$ is $\dclL(B^\Theta)$-independent, as $\jet^\infty(a)^{\Theta_0} = a^\Theta$. Thus, we may apply Proposition~\ref{ispregeom} to $(X,\cl) = (\cM,\clDz)$ to deduce that $(\cM,\clD)$ is a finitary matroid.
\end{proof}

We can leverage this to define a dimension function as in \S\ref{subsec:dim}. Let $(\M,\Delta)$ be a monster model of $\TDG$ and suppose that $(\cM,\Delta)$ is a small elementary substructure of $(\M,\Delta)$.

\begin{lemma}\label{cldelemD}
Let $B \subseteq \M$. If $\clD(B) = B$ then $(B,\Delta|_B) \models \TDG$, where $\Delta|_B = \{\delta|_B:\delta \in \Delta\}$.
\end{lemma}
\begin{proof}
Since $\dclL(B) \subseteq \clD(B) = B$, we have that $B \preceq_L \M$. Since $\delta B \subseteq \clD(B) = B$ for each $\delta \in \Delta$, we have that $(B,\Delta|_B)\models \Td$. Let $\cC = \big(P,U,(f_\beta)_{\beta \in P}\big)$ be a coherent condition on $B$. 
We first consider the case that $P \neq\0$. Since $(\M,\Delta) \models \TDG$, there is some $a \in \M$ which satisfies $\cC$. If $\beta \in P$, then $a^\beta = f_\beta(a^I)$ where $I$ is as in \S\ref{subsec:coherent}. Thus, $a^\Theta$ is not $\dclL(B)$-independent so $a \in \clD(B) = B$. 

\medskip

Now consider the case that $P = \0$. Since $U$ is an open $L(B)$-definable subset of $\M^\Theta$ there is some $<$-closed subset $J \subseteq \Theta$ and some open $L(B)$-definable subset $V \subseteq \M^J$ such that $U = V \times \M^{\Theta\setminus J}$. Let $\beta$ be the $<$-minimal element of $\Theta\setminus J$, set $f_\beta := 0$ and set $\cC' := \big(\{\beta\},V,f_\beta\big)$. Since $\bigvee\{\beta\} = \beta$, we have that $\cC'$ is coherent. By the previous case, $\cC'$ is satisfied by some $a \in B$. Then $a^\Theta \in U$, so $a$ satisfies $\cC$.
\end{proof}

Let $\rkD$ be the rank function associated to $\clD$. By the proof of Proposition~\ref{existential} (with Lemma~\ref{cldelemD} in place of Lemma~\ref{cldelem}) and by the remarks after Definition~\ref{deltadim}, we have the following.
\begin{proposition}
$(\M,\clD)$ is an existential matroid in the sense of~\cite{Fo11}. Thus, we have a dimension function on the algebra of $\LD(M)$-definable sets given by
\[
\dimD(A)\ :=\ \max\big\{\rkD(\av|M):\av \in A\big\}
\]
for each nonempty $\LD(M)$-definable set $A\subseteq \M^n$. This dimension function satisfies the axioms in~\cite{vdD89}.
\end{proposition}

Using this dimension function, one can make the obvious changes in \S\ref{subsec:opencore} to show the following:

\begin{proposition}
$\TDG$ has $T$ as its open core. More precisely, for $B\subseteq \M$, any open $\LD(B)$-definable set is $L(B^\Theta)$-definable.
\end{proposition}

\appendix
\section{\texorpdfstring{$\cC^k$}{Ck}-cells and \texorpdfstring{$\cC^k$}{Ck}-functions} \label{sec:Ck}
In this section, we establish a fiberwise result about definable families of $\cC^k$-functions, which generalizes Corollary 6.2.4 in~\cite{vdD98}. Given a $\cC^1$-manifold $X$ and a point $\cv \in X$, we let $T_{\cv} X$ denote the tangent space of $X$ at $\cv$.

\medskip 

A $\cC^k$-cell is a special type of definable $\cC^k$-submanifold of $M^n$ with an associated binary sequence $(i_1,\ldots,i_n) \in \{0,1\}^n$. The cells and their sequences are defined by induction on $n$:
\begin{enumerate}[(i)]
\item A $(1)$-cell in $M$ is an open interval and a $(0)$-cell is a singleton.
\item Given an $(i_1,\ldots,i_n)$-cell $D \subseteq M^n$ and an $L(M)$-definable $\cC^k$-function $f:D \to M$, $\Gamma(f)$ is an $(i_1,\ldots,i_n,0)$-cell and the following are $(i_1,\ldots,i_n,1)$-cells:
\begin{itemize}
\item $\big\{(\x,y) \in D \times M:y<f(\x)\big\}$;
\item $\big\{(\x,y) \in D \times M:y>f(\x)\big\}$;
\item $D \times M$;
\end{itemize}
Given an $L(M)$-definable $\cC^k$-function $g:D \to M$ with $f(\x)<g(\x)$ on $D$, the set 
\[
\big\{(\x,y) \in D \times M:f(\x)<y<g(\x)\big\}
\]
is also an $(i_1,\ldots,i_n,1)$-cell. 
\end{enumerate}
Note that a $\cC^k$-cell is open if and only if it is a $(1,\ldots,1)$-cell. We call the binary sequence associated to a $\cC^k$-cell $D$ the \textbf{type} of $D$. We refer to $\cC^0$-cells just as cells.

\medskip 

The inductive construction of $\cC^k$-cells makes them very easy to work with. For example, the next lemma fails for $\cC^1$-submanifolds in general, but it holds for $\cC^1$-cells.
\begin{lemma}\label{tanspace}
Let $D \subseteq M^{m+n}$ be a $\cC^1$-cell. Then 
\[
\pi_m(T_{\av,\bv} D)\ =\ T_\av(\pi_m D)
\]
for all $(\av,\bv) \in D$.
\end{lemma}
\begin{proof}
We first handle the case $n =1$. Fix $\av \in M^m$ and $b \in M$ with $(\av,b)\in D$ and set $D':= \pi_m D $. Let $(i_1,\ldots,i_m,i_{m+1})$ be the type of $D$. If $i_{m+1} = 1$, then $T_{\av,b}D = T_{\av} D' \times M$, proving the lemma. If $i_{m+1} = 0$, then $D = \Gamma(g)$ for some $L(M)$-definable $\cC^1$-function $g:D' \to M$. Take an $L(M)$-definable $\cC^1$-function $G:U\to M$ with $U \supseteq D'$ open and with $G|_{D'} = g$. Let $P\subseteq M^{m+1}$ be hyperplane 
\[
\big\{(\x,y) \in M^{m+1}:y = \Jac_G(\av) \x\big\}.
\]
Then $T_{\av,\bv}D = (T_{\av} D' \times M) \cap P$ and $\pi_m(T_{\av,\bv} D) =T_\av D'$ as desired. The general case follows easily by induction on $n$.
\end{proof}

One of the most useful tools in the study of o-minimal fields is the $\cC^k$-cell decomposition theorem below. See~\cite{vdD98} for the cases $k = 0$ or 1. A \textbf{$\cC^k$-cell decomposition of $M^n$} is a finite collection $\cD$ of disjoint $\cC^k$-cells such that $\bigcup\cD = M^n$ and such that $\{\pi_{n-1}D:D \in \cD\}$ is a $\cC^k$-cell decomposition of $M^{n-1}$.

\begin{proposition}[$\cC^k$-cell decomposition]\
\begin{enumerate}[(i)]
\item For any $L(M)$-definable sets $A_1,\ldots,A_p \subseteq M^n$ there is a $\cC^k$-cell decomposition $\cD$ of $M^n$ partitioning $A_1,\ldots,A_p$, i.e. each $D\in \cD$ is disjoint from or contained in each $A_i$.
\item For every $L(M)$-definable map $f:A\to M^m$ with $A\subseteq M^n$, there is a $\cC^k$-cell decomposition of $M^n$ partitioning $A$ such that the restriction $f|_D$ is $\cC^1$ for each cell $D \in \cD$ contained in $A$.
\item Given an $L(M)$-definable map $f:A\to M^m$ with $A\subseteq M^n$, let $A' := \big\{\x \in A: \Jac_{f} \text{ is defined at }\x\big\}$. Then $A \setminus A'$ has empty interior.
\end{enumerate}
\end{proposition}

If $\cD$ is a cell decomposition as in (i), we say that \textbf{$\cD$ partitions $A_1,\ldots,A_p$} and if $\cD$ is a cell decomposition as in (ii), we say that \textbf{$\cD$ is a $\cC^k$-cell decomposition for $f$}. By taking refinements, we can always find a $\cC^k$-cell decomposition for $f$ which partitions $A_1,\ldots,A_p$.
Suppose that $f$ and each $A_i$ are $L(B)$-definable for some $B \subseteq M$. Then by passing to the elementary substructure with universe $\dclL(B)$, we see that we can take an $L(B)$-definable $\cC^k$-cell decomposition $\cD$ for $f$ which partitions $A_1,\ldots,A_p$ (i.e. each cell is $L(B)$-definable).

\subsection{Definable families of $\cC^k$-functions} In this subsection, fix $B \subseteq M$ and an $L(B)$-definable function $F:U \to M$ where $U \subseteq M^{m+n}$. Set $U' := \pi_m (U)$ and suppose that $U_{\av}$ is open in $M^n$ for each $\av \in M^m$.

\begin{lemma}\label{lem:C1-generic}
Suppose that there is a $\dclL(B)$-independent tuple $\av \in U'$ such that $F_{\av}:U_{\av}\to M$ is $\cC^1$. Then there is an $L(B)$-definable open cell $D \subseteq U'$ containing $\av$ such that $F|_{U \cap (D\times M^n)}$ is $\cC^1$.
\end{lemma}
\begin{proof}
We view $F$ as a function of the variables $\x = (x_1,\ldots,x_m)$ and $\y = (y_1,\ldots,y_n)$.
Set
\[
A\ :=\ \{\x \in U':F_{\x}\text{ is } \cC^1 \text{ on }U_{\x}\}.
\]
Then $\av \in A$, so by~\cite[Corollary 6.2.4]{vdD98},
there is a definable set $A' \subseteq A$ containing $\av$ such that the function $F$ and the map $\frac{\partial F}{\partial \y}$ are continuous on $U\cap(A'\times M^n)$.
Take an $L(B)$-definable $\cC^1$-cell decomposition $\cD$ for $F$ partitioning $A' \times M^n$.
Let $\cD' \subseteq \cD$ be the cells in $\cD$ which are contained in $U$ and which intersect $\{\av\} \times M^n$. We let $D$ be the common projection of these cells onto $M^m$ and we claim that $D$ satisfies the conditions in the statement of the lemma. Since $D$ contains the independent tuple $\av$, it must be open and contained in $A'$, so both $F$ and $\frac{\partial F}{\partial \y}$ are continuous on $U \cap (D\times M^n)$. 
It remains to show that the map $\frac{\partial F}{\partial \x}$ is continuous on $U\cap(D\times M^n)$.

\medskip

For each $d \in \{0,1,\ldots,n\}$, we let $\cD'_d$ be the set of all cells in $\cD$ which have
codimension at least $d$ in the ambient space $M^{m+n}$. We set $U_d:=\bigcup
\cD'_d$. We remark that $U_d$ is open for each $d$ and that $U_n = \bigcup \cD' = U\cap(D \times
M^n)$. We will show by induction on $d$ that $\frac{\partial F}{\partial \x}$ is continuous
on $U_d$.
The $d = 0$ case follows by our choice of cell decomposition. Fix $d > 0$ and
suppose that $\frac{\partial F}{\partial \x}$ is continuous on $U_{d-1}$.
Let $C \in \cD'$ be a cell of codimension $d$.
After a permutation of variables, we may assume that $C' :=\pi_{m+n-d}(C)$ is open and that $C$ is of the form
\[
C\ =\ \big\{(\x,\y): (\x,\y') \in C' \text{ and } \y''= G(\x,\y')\big\}
\]
where $\y' = (y_1,\ldots,y_{n-d})$, $\y'' = (y_{n-d+1},\ldots,y_n)$ and where $G:C' \to
M^{d}$ is an $L(B)$-definable $\cC^1$-map. We will show that $\frac{\partial F}{\partial \x}$ is
continuous at each point in $C$. 
Note that any point in $C$ is contained in a small open ball which only intersects $C$ and cells of codimension larger than $d$.

\medskip

Define the function $\tilde{F}$ by
\[
\tilde{F}(\x,\y)\ :=\ F\big(\x,\y',\y''+G(\x,\y')\big)-F\big(\x,\y',G(\x,\y')\big).
\]
This function is defined on $C' \times (-\epsilon,\epsilon)^d$ for some sufficiently small positive
$\epsilon \in M$. By replacing $F$ by $\tilde{F}$, we may assume that 
\[
C\ =\ \big\{(\x,\y): (\x,\y') \in C' \text{ and } \y''= \zero\big\}
\]
and that the restrictions of $F$, $\frac{\partial F}{\partial \x}$ and $\frac{\partial F}{\partial \y'}$ to $C$ are all identically zero.
Thus, it remains to show that 
\[
\lim\limits_{\y'' \to \zero}\frac{\partial F}{\partial \x}(\x,\y',\y'')\ =\ \zero
\]
for all $(\x,\y') \in C'$. By~\cite[Lemma 6.4.2]{vdD98}, it suffices to show that 
\[
\lim\limits_{t \to 0}\frac{\partial F}{\partial \x}\big(\x,\y',\gamma(t)\big)\ =\ \zero
\]
for an arbitrary $L(B)$-definable curve $\gamma:(0,1) \to (-\epsilon,\epsilon)^d$ with limit $\gamma(t) \to \zero$ as $t \to 0$. Fix $(\av, \bv') \in C'$ and for each $t$, set
\[
f_t\ :=\ F\big(\av,\bv',\gamma(t)\big).
\]
By~\cite{Lo98} there exists an $L(B)$-definable $\cC^2$ Verdier stratification $\cV$ of $M^{m+n+1}$ which is compatible with both $\Gamma(F)$ and with $C \times M$
(Loi works in an o-minimal expansion of $\R$, but his proof generalizes with minor changes to any o-minimal structure expanding an ordered field).
Let $X \in \cV$ be the submanifold containing $(\av, \bv', \zero, 0)$ and
$X' \in \cV$ be the submanifold such that $\big(\av, \bv', \gamma(t), f_t\big) \in X'$ for $t$ sufficiently small.
Note that $X \subseteq \overline{X'} \setminus X'$ and that
\[
T_{a,b', \gamma(t),f_t} X'\ \subseteq\ \Gamma(\D_t)
\]
where $\D_t:M^{m+n}\to M$ is the linear function given by 
\[
\D_t(\uv,\vv)\ =\ \frac{\partial F}{\partial \x}\big(\av,\bv',\gamma(t)\big) \uv+\frac{\partial F}{\partial \y}\big(\av,\bv',\gamma(t)\big) \vv.
\]

\medskip

Since the projection $\pi_m X$ contains $\av$, it must be open, so $T_{\av} ( \pi_{m}X) = M^m$. By Lemma~\ref{tanspace}, we have that $\pi_{m}(T_{\av, \bv', \zero,0} X) =T_{\av} ( \pi_{m}X)$, so $(T_{\av, \bv', \zero,0} X)_\uv \neq \0$ for each $\uv \in M^m$.
Let $\uv_0 \in M^m$ be arbitrary and take $\lambda \in M^{\neq}$ and $\vv_0' \in M^{n-d}$ such that
$(\lambda \uv_0, \vv_0', \zero,0) \in T_{\av, \bv', \zero,0} X$ and such that $\big\Vert (\lambda \uv_0, \vv_0', \zero,0)\big\Vert = 1$.
By the Verdier condition, we have that
\[
\lim\limits_{t\to 0}\ \delta(T_{\av, \bv', \zero, 0} X, T_{\av,\bv', \gamma(t),f_t} X')\ =\ 0,
\]
where 
\[
\delta(V,V')\ =\ \sup_{\vv\in V,\Vert \vv\Vert = 1}d(\vv,V)
\]
is the distance between vector subspaces $V,V'\subseteq M^n$.
Thus, for every sufficiently small $t$, we can find $\uv_t \in M^m$ and $\vv_t \in M^n$ such that $\big\Vert \big(\uv_t,\vv_t,\D_t(\uv_t,\vv_t)\big)\big\Vert = 1$ and such that as $t \to 0$, we have
\[
\uv_t \to \lambda \uv_0,\qquad \vv_t \to (\vv_0',\zero),\qquad \D_t(\uv_t,\vv_t) \to 0.
\]
Note that 
\[
\lim\limits_{t \to 0}\frac{\partial F}{\partial \y}\big(\av,\bv',\gamma(t)\big) \vv_t\ =\ \frac{\partial F}{\partial \y}(\av,\bv',\zero) (\vv_0',\zero)\ =\ \frac{\partial F}{\partial \y'}(\av,\bv',\zero) \vv_0\ =\ \zero,
\]
so we have
\[
\lim\limits_{t\to 0}\D_t(\uv_t,\vv_t)\ =\ \lim\limits_{t\to 0} \frac{\partial F}{\partial \x}\big(\av,\bv',\gamma(t)\big) \uv_t\ =\ \left(\lim\limits_{t\to 0} \frac{\partial F}{\partial \x}\big(\av,\bv',\gamma(t)\big)\right) \lambda \uv_0\ =\ \zero.
\]
Since $\uv_0$ is arbitrary, this shows that $ \frac{\partial F}{\partial \x}\big(\av,\bv',\gamma(t)\big) \to \zero$ as $t \to 0$.
\end{proof}

\begin{corollary}\label{Ckfiber}
Suppose that $F_{\av}$ is $\cC^k$ on $U_{\av}$ for all $\av \in U'$. Then there exists an $L(B)$-definable $\cC^k$-cell decomposition
$\cD$ of $M^m$ such that $F|_{U\cap(D\times M)}$ is $\cC^k$ for each $D \in \cD$.
\end{corollary}
\begin{proof}
This follows from~\cite[Corollary 6.2.4]{vdD98} if $k =0$, so we assume $k \ge 1$.
We proceed by induction on $m$.
If $m = 0$, the result is clear.
Assume now that $m > 0$ and that we have already proved the result for every
$m' < m$.
Define
\[
A\ :=\ \{\x \in U': F|_{U\cap (V \times M^n)} \text{ is } \cC^k \text{ for some open neighborhood }V \text{ of }\x\}.
\]
By Lemma~\ref{lem:C1-generic} applied
to $F$ and all its derivatives of order $\leq k-1$, we see that $A$ is $L(B)$-definable and that $\dimL(U'\setminus A) < m$.
Let $\tilde{\cD}$ be a $\cC^k$-cell
decomposition for $F$ partitioning $A$. If $D\in \tilde{\cD}$ is contained in $A$, then $F|_{U\cap(D\times M)}$ is $\cC^k$ by definition. If $D \in \tilde{\cD}$ is disjoint from $A$, then set $d := \dimL(D)$ and fix an $L(M)$-definable $\cC^k$-diffeomorphism $g:M^d \to D$. Set $U_D:=\big\{(\x,\y)\in M^{d+n}:\y \in U_{g(\x)}\big\}$ and define $H: U_D\to M$ by
\[
H(\x, \y)\ :=\ F\big(g(\x), \y\big).
\]
Since $d < m$, we may apply our induction hypothesis to $H$ and take an $L(B)$-definable $\cC^k$-cell decomposition $\{D_1,\ldots,D_p\}$ of $M^d$ such that $H|_{U_D\cap(D_i\times M^n)} = F|_{U\cap(g(D_i) \times M^n)}$ is $\cC^k$. We refine $\tilde{\cD}$ by replacing $D$ with $\cC^k$-cells refining $g(D_1),\ldots,g(D_p)$. Repeating this process for each cell $D \in \tilde{\cD}$ which is not contained in $U$, we arrive at the promised decomposition $\cD$.
\end{proof}


\end{document}